\pdfoutput=1 %When you use pdflatex
\documentclass{amsart} %When you use pdflatex

\usepackage{tikz}
\usetikzlibrary{calc}
\usetikzlibrary{decorations.markings}
\usetikzlibrary {arrows.meta}
\usetikzlibrary {bending}
\usetikzlibrary{intersections}

\usepackage{amssymb}
\usepackage{amsmath}
\usepackage{url}
\usepackage{enumerate}
\usepackage{graphicx}

\usepackage{hyperref}
\hypersetup{
 colorlinks = true,
 linkcolor = blue,
 filecolor = blue,
 urlcolor = blue,
 citecolor = blue,
}
\hypersetup{unicode=true}

\usepackage{bm}
\usepackage{enumitem}

\theoremstyle{plain}
\newtheorem{theorem}{Theorem}[section]
\newtheorem{lemma}[theorem]{Lemma}
\newtheorem{proposition}[theorem]{Proposition}
\newtheorem{corollary}[theorem]{Corollary}

\theoremstyle{definition}
\newtheorem{definition}[theorem]{Definition}

\theoremstyle{remark}
\newtheorem{remark}[theorem]{Remark}

\numberwithin{equation}{section}
\numberwithin{table}{section}
\numberwithin{figure}{section}
\renewcommand{\qed}{\hfill {$\Box$}}

\newcommand{\CC}{\mathord{\mathbb C}}

\newcommand{\EE}{\mathord{\mathbb E}}
\newcommand{\FF}{\mathord{\mathbb F}}

\newcommand{\PP}{\mathord{\mathbb  P}}

\newcommand{\ZZ}{\mathord{\mathbb Z}}

\newcommand{\GGG}{\mathord{\mathcal G}}
\newcommand{\HHH}{\mathord{\mathcal H}}
\newcommand{\III}{\mathord{\mathcal I}}

\newcommand{\LLL}{\mathord{\mathcal L}}

\newcommand{\OOO}{\mathord{\mathcal O}}
\newcommand{\PPP}{\mathord{\mathcal P}}

\newcommand{\SSS}{\mathord{\mathcal S}}
\newcommand{\TTT}{\mathord{\mathcal T}}
\newcommand{\UUU}{\mathord{\mathcal U}}

\newcommand{\XXX}{\mathord{\mathcal X}}

\newcommand{\ZZZ}{\mathord{\mathcal Z}}

\newcommand{\maprightsp}[1]{\; \smash{\mathop{\; \longrightarrow \; }\limits\sp{#1}}\; }
\newcommand{\mapleftsp}[1]{\; \smash{\mathop{\; \longleftarrow \; }\limits\sp{#1}}\; }

\newcommand{\mapdown}{\phantom{\Big\downarrow}\hskip -8pt \downarrow}
\newcommand{\mapdownleft}[1]{\llap{$\vcenter{\hbox{$\scriptstyle#1$\;}}$}\mapdown}

\newcommand{\inj}{\hookrightarrow}
\newcommand{\surj}{\mathbin{\to \hskip -7pt \to}}
\newcommand{\isom}{\xrightarrow{\raise -3pt \hbox{\scriptsize $\sim$} }}

\newcommand{\set}[2]{\{\,{#1}\mid {#2} \,\}}
\newcommand{\bigset}[2]{\left\{\; {#1} \; \left\vert \; {#2} \;  \right.\right \}}

\newcommand{\angs}[1]{\langle {#1}  \rangle}

\newcommand{\inv}{\sp{-1}}
\newcommand{\dual}{\sp{\vee}}

\newcommand{\sprime}{\sp{\prime}}
\newcommand{\sperp}{\sp{\perp}}

\DeclareMathOperator{\Aut}{Aut}
\DeclareMathOperator{\Image}{Image}
\DeclareMathOperator{\Sing}{Sing}
\DeclareMathOperator{\Ker}{Ker}

\newcommand{\PGL}{\mathord{\mathrm{PGL}}}
\newcommand{\OG}{\mathord{\mathrm{O}}}

\newcommand{\id}{\mathord{\mathrm{id}}}

\newcommand{\mystruth}[1]{\phantom{\hbox{\vrule height #1}}}

\newcommand{\pione}{\pi_1}

\newcommand{\intf}[1]{\langle #1 \rangle}

\newcommand{\spset}[1]{^{\{#1\}}}

\newcommand{\mtan}[1]{t_{#1}}
\newcommand{\ADE}{\mathrm{ADE}}
\newcommand{\Tac}{\mathrm{Tac\,}}
\newcommand{\Pic}{\mathrm{Pic}}

\newcommand{\mon}{\Phi}
\newcommand{\LLn}[1]{\LLL^{[#1]}}
\newcommand{\bfP}{\mathbf{P}}

\newcommand{\blambda}{\bm{\lambda}}
\newcommand{\bp}{\bm{p}}
\newcommand{\bP}{\bm{P}}
\newcommand{\brho}{\bm{\rho}}

\newcommand{\sfP}{\mathsf{P}}
\newcommand{\sfB}{\mathsf{B}}
\newcommand{\sfQ}{\mathsf{Q}}

\newcommand{\birat}{\dashrightarrow}

\newcommand{\tilQ}{\widetilde{Q}}
\newcommand{\tilB}{\widetilde{B}}

\newcommand{\barW}{\overline{W}}
\newcommand{\barDelta}{\overline{\Delta}}
\newcommand{\barL}{\overline{L}}
\newcommand{\barH}{\overline{H}}

\newcommand{\involB}{i_B}

\begin{document}

\title[Del Pezzo surfaces and Zariski multiples]
{Del Pezzo surfaces of degree one and examples of Zariski multiples}

\author{Ichiro Shimada}
\address{Department of Mathematics,
Graduate School of Science,
Hiroshima University,
1-3-1 Kagamiyama,
Higashi-Hiroshima,
739-8526 JAPAN}
\email{ichiro-shimada@hiroshima-u.ac.jp}
\thanks{Supported by JSPS KAKENHI Grant Number~20H00112, 23K20209, and 23H00081}

%超平面配置に関連する離散構造の拡張、深化とその応用
%23H00081
%代数幾何学の計算機による研究の新展開
%23K20209
%格子、保型形式とK3曲面、エンリケス曲面の研究
%20H00112

\begin{abstract}
We construct examples of Zariski 
$N$-tuples with large 
$N$ using the monodromy action 
of the Weyl group of type $E_8$ on the set of $240$ lines in a del Pezzo surface of degree one.
\end{abstract}
\maketitle
\section{Introduction}
We work over the field of complex numbers.
A \emph{plane curve} means a reduced, possibly reducible, projective plane curve.
\par
Zariski~\cite{ZariskiA, ZariskiB} demonstrated that an equisingular family of plane curves may fail to be connected by 
providing an example of a pair of $6$-cuspidal plane sextics
such that  their complements %in the projective plane 
have non-isomorphic fundamental groups.
Artal~Bartolo~\cite{Artal1994}  revisited Zariski's work, 
and defined  
a \emph{Zariski pair} 
as a pair of plane curves
 that have the same combinatorial type of singularities 
 but differ in their embedding topologies. 
 More generally, a collection of 
$N$ plane curves is called a \emph{Zariski  $N$-tuple}
 if any two  in the collection form a Zariski pair.
 \par
Since~\cite{Artal1994}, many Zariski multiples have been constructed 
using a wide variety of methods.
The construction of Zariski multiples %seems to have served as 
has served as 
a good testing ground for various techniques in the study of embedding topology of plane curves.
See, for example,  the survey paper~\cite{TheSurvey}.
\par
In this paper,
we employ the lattice structure of the middle homology group of the double plane 
branching along the plane curve as an invariant to distinguish topological types~\cite{Shimada2010}.
We enumerate the topological types in our Zariski multiples  by a monodromy argument initiated in~\cite{Shimada2003}.
Del Pezzo surfaces of degree~$1$ and the Weyl group of type~$E_8$ play 
an important role in our investigation.
\par
As the main result, we obtain the following:
\begin{theorem}\label{thm:main1}
For each  integer $k$ satisfying $1<k<119$,
there exists a Zariski $N(k)$-tuple $\ZZZ_k$
consisting of plane curves of degree $7+2k$,
where 
\begin{equation}\label{eq:Nbinom}
N(k)\ge \frac{1}{348364800}\binom{120}{k}.
\end{equation}
\end{theorem}
See Section~\ref{sec:main} for the precise description 
of the combinatorial type 
of the plane curves in  $\ZZZ_k$, and 
the exact values of $N(k)$.
We have $N(120-k)=N(k)$, and 
the values of $N(k)$ for small $k$ are as follows:
%
%
%gap> Read("Task250610a.txt");
%1 1 
%2 2 
%3 5 
%4 15 
%5 48 
%6 212 
%7 1116 
%8 7388 
%9 56946 
%
%
\begin{equation}\label{eq:Nks}
\begin{array}{c|ccccccccc}
k & 1 & 2 & 3 & 4 & 5 & 6 &7 &8 &9 \\
\hline 
N(k) & 1& 2 & 5 & 15 & 48 & 212 &1116& 7388& 56946\rlap{\;.}
\end{array}
\end{equation}
Putting $k=60$,
we obtain an example of a  Zariski $N(60)$-tuple with 
\begin{equation*}\label{eq:largeN}
N(60)>2.77\times 10^{26}.
\end{equation*}
\par
In our previous paper~\cite{Shimada2022},
we constructed  Zariski multiples by
using the monodromy of the Weyl group of type $E_7$
for a family of del Pezzo surfaces of degree $2$.
For example, we obtained a Zariski $105$-tuple 
in which each member is 
a union of a smooth quartic curve and $14$ lines chosen from its $28$ bitangents.
These examples extended the works~\cite{Bannai2016, Bannai2019, Bannai2020}
by Bannai et al.
We also constructed a series of  Zariski $N$-tuple with $N\to \infty$  by 
means of
 $4$-tangent conics of a smooth quartic curve.
\par
It is a natural step to extend this argument to 
the  Weyl group of type $E_8$ and 
del Pezzo surfaces of degree $1$.
A smooth quartic curve  in~\cite{Shimada2022} is now replaced by a $\mtan{3}$-sextic~(Definition~\ref{def:t3sextic}),
and its $28$ bitangents are replaced by $120$ special tangent conics~(Definition~\ref{def:specialtangent}).
The choices of $k$ conics from the $120$ special tangent conics yield
the Zariski $N(k)$-tuple $\ZZZ_k$.
It seems to be an interesting problem 
 to ask what becomes of the $4$-tangent conics of a smooth quartic curve in the context of a $\mtan{3}$-sextic.
 See~Remark~\ref{rem:K3}.
\par
This paper contains two improvements compared with~\cite{Shimada2022}.
One is that we present a  unified method to compute the monodromy group 
for families of del Pezzo surfaces.
In fact, the proof of~\cite[Theorem~3.1]{Shimada2022} on the monodromy of a family of  del Pezzo surfaces of degree  $2$ 
was incomplete, and we give a rectified argument in the present paper.
Another improvement is that 
we use a simpler reasoning to distinguish topological types in our Zariski multiples.
This simplification is based on an observation (Lemma~\ref{lem:AB})  by Artal~Bartolo.
\par
In several parts of the proofs, we rely on brute-force computations carried out by a computer. 
For this purpose, we employ {\tt GAP}~\cite{GAP}.
A detailed computational data is available from~\cite{WE8compdata}.
\par
The plan of this paper is as follows.
In Section~\ref{sec:main}, we precisely state our main result in Theorem~\ref{thm:main2}.
We define the combinatorial type %$\sigma_k$ 
of our  Zariski $N(k)$-tuples, 
and give the exact value of $N(k)$.
In Section~\ref{sec:monodromy}, 
we develop a general theory to compute the monodromy group
for families of del Pezzo surfaces. 
In Section~\ref{sec:lines}, 
we study the lines in a del Pezzo surface of degree $1$
more closely.
In Section~\ref{sec:mtan3anddelPezzo}, 
we relate the theory of del Pezzo surfaces to that of singular plane curves,
and deduce some properties of $\mtan{3}$-sextics
from our discussion about  del Pezzo surfaces of degree $1$.
In Section~\ref{sec:embtop}, 
we investigate the embedding topology of $\mtan{3}$-sextics and 
prove Theorem~\ref{thm:main2}.
We conclude this paper by a remark on the work~\cite{Roulleau2022} about $K3$ surfaces obtained as 
double covers of  del Pezzo surfaces of degree $1$.
%
%
% notation
\subsection*{Notation}
\begin{enumerate}[label={(\arabic*)}]
\item For a set $M$ and a non-negative integer $k$,
we denote  by 
$M\spset{k}:=\displaystyle{\binom{M}{k}}$
 the set of subsets $S\subset M$ of size $|S|=k$.
\item The orthogonal group $\OG(L)$ of a lattice $L$ acts on $L$ from the right.
\item For a topological space $T$, 
we write $H_i(T)$ and $H^i(T)$ for $H_i(T;\ZZ)$ and $H^i(T;\ZZ)$, respectively.
\end{enumerate}
%
% acknowledgements
\subsection*{Acknowledgements}
The author expresses his sincere gratitude to 
Professor Enrique Artal~Bartolo for providing Lemma~\ref{lem:AB} and for many valuable discussions.
He also thanks 
Professor Shinzo Bannai, Professor Akira Ohbuchi, and Professor Masahiko Yoshinaga
for many comments and discussions.
\section{Main result}\label{sec:main}
In Section~\ref{subsec:sigmak},
we define the combinatorial type $\sigma_k$ 
of plane curves in our Zariski $N(k)$-tuple $\ZZZ_k$.
In  Section~\ref{subsec:Nk}, 
we define the number $N(k)$,  and state 
our main result Theorem~\ref{thm:main2}.
\subsection{Combinatorial type 
\texorpdfstring{$\sigma_k$}{sigmak}}\label{subsec:sigmak}
\begin{definition}\label{def:combtype}
Two plane curves $D$ and $D\sprime$ are said to 
have the \emph{same combinatorial type}
if there exist  tubular neighborhoods $T\subset \PP^2$ of $D$ and 
$T\sprime\subset \PP^2$ of $D\sprime$
such that 
$(\PP^2, T)$ and $(\PP^2, T\sprime)$ are homeomorphic.
\end{definition}
See~\cite[Remark 3]{TheSurvey} for another formulation of the notion of combinatorial type.
%See~\cite[Remark 3]{TheSurvey} for the precise definition of the notion of 
%\emph{combinatorial type  of plane curves}.
%
\begin{definition}\label{def:mtan3}
A germ  $(C, \mathbf{0})$ of isolated plane curve singularity
is said to be a \emph{$\mtan{m}$-singularity} if 
$C$ consists of $m$ smooth local branches
and each pair of the local  branches has intersection number $2$.
\end{definition}
Note that a $\mtan{2}$-singularity is an ordinary tacnode (an $a_3$-singularity).
\begin{definition}\label{def:ellA}
Let $C\subset \PP^2$ be a plane curve with a $\mtan{m}$-singularity at $A\in C$.
The common tangent line $\Lambda\subset \PP^2$ to the local branches of $C$ at $A$
is called the \emph{tangent line to $C$ at $A$}.
\end{definition}
\begin{definition}\label{def:t3sextic}
A plane curve $C$  of degree $6$ is 
called a  \emph{$\mtan{3}$-sextic}  
if the singular locus of $C$ consists of a single point $A$, 
and  $A\in C$ is a $\mtan{3}$-singular point.
\end{definition}
Note that a $\mtan{3}$-sextic
 is irreducible.
We fix a point $A\in \PP^2$ and a line $\Lambda\subset \PP^2$
passing through $A$.
\begin{definition}
A $\mtan{3}$-sextic 
with the singular point $A$ and the tangent line $\Lambda$ at $A$
is called a $\mtan{3}$-sextic  \emph{in the frame $(A, \Lambda)$}.
\end{definition}
Let $C$ be a $\mtan{3}$-sextic in the frame $(A, \Lambda)$.
\begin{definition}\label{def:specialtangent}
A smooth conic $\Gamma$ is said to be a \emph{special tangent conic} of $C$
if the following  hold; 
\begin{itemize}
\item 
the conic $\Gamma$ passes through  $A$, and 
$C+\Gamma$ has  $\mtan{4}$-singularity at $A$, and 
 \item
at every intersection point of $C$ and $\Gamma$ other than $A$,
the intersection multiplicity is even.
 \end{itemize}
 \end{definition}
 For a special tangent conic $\Gamma$, 
 we put 
 \[
 \Tac(\Gamma):=\Sing(C+\Gamma)\setminus\{A\}.
 \] 
\begin{definition}\label{def:generalset}
We say that 
a set $\{\Gamma_1, \dots, \Gamma_{k}\}$  
consisting of $k$  special tangent conics of $C$  
is said to be \emph{``in a general position"} 
if the following conditions hold.
\begin{itemize}
\item 
Any two of $\Gamma_1, \dots, \Gamma_{k}$ have local intersection number $2$ at $A$.
\item 
Each $ \Tac(\Gamma_i)$ consists of $3$ tacnodes of $C+\Gamma_i$.
\item 
The sets $\Tac(\Gamma_1), \dots, \Tac(\Gamma_{k})$ 
are disjoint to each other.
 \item
 The singular points  of 
 the union $C+\Gamma_1+\dots +\Gamma_{k}$ 
 other than $A$ and the tacnodes in $\Tac(\Gamma_1), \dots, \Tac(\Gamma_{k})$ 
 are ordinary nodes.
 \end{itemize}
\end{definition}
In Section~\ref{subsec:planecurveswithmtan}, we  prove the following:
\begin{proposition}\label{prop:t3}
All  $\mtan{3}$-sextics 
 in the frame $(A, \Lambda)$
are parameterized by a Zariski open subset  $\TTT$
of a $15$-dimensional linear subspace of $|\OOO_{\PP^2}(6)|$.
\end{proposition}
For a point $t\in \TTT$,
we denote by $C_t$ the corresponding $\mtan{3}$-sextic.
In Section~\ref{subsec:pip}, we  prove the following:
\begin{proposition}\label{prop:t3conics}
Let $t$ be a general point of $\TTT$.
Then $C_t$ 
 has exactly $120$  special tangent conics, and they 
are in a general position.
\end{proposition}
For $t\in \TTT$,
we denote by $G(C_t)$ the set of special tangent conics of $C_t$.
For $s=\{\Gamma_1, \dots, \Gamma_k\}\in G(C_t)\spset{k}$,
we put
\[
D_{t, s}:=C_t+\Lambda+\Gamma_1+\cdots+\Gamma_k.
\]
\begin{definition}\label{def:sigmak}
By Proposition~\ref{prop:t3conics},
the combinatorial type
of $D_{t,s}$ does not depend on the choice of 
$t\in \TTT$ and $s \in G(C_t)\spset{k}$,
provided that $t$ is general in $\TTT$.
We denote this combinatorial type 
by $\sigma_k$.
\end{definition}
%
%\begin{remark}
The curve of combinatorial type $\sigma_k$  
is of degree $7+2k$,  and its singularities consist
of one $\mtan{4+k}$-singular point, $3k$ tacnodes, and $k(k-1)$ ordinary nodes.
It should be noted that the information given by 
a combinatorial type 
 includes,  not only the types of singular points, 
but also more detailed data such as which irreducible components 
correspond to which local branch of each singular point.
See~\cite[Remark 3]{TheSurvey}.
%\end{remark}
%
\subsection{Main Theorem}\label{subsec:Nk}
Let $\EE_8$ be the root lattice of type $E_8$,
that is, $\EE_8$ is an even unimodular \emph{negative-definite} lattice of rank $8$
generated by vectors of square norm $-2$.
(Note that we adopt the sign convention opposite to the standard one.)
We denote by $\Delta(\EE_8)$ be the set of vectors of square norm $-2$ in $\EE_8$,
which is of size $240$.
Let $W(\EE_8)$ be the Weyl group of $\EE_8$,
that is, the subgroup of $\OG(\EE_8)$ generated by reflections with respect to vectors
of square norm $-2$. 
In fact, we have an equality $W(\EE_8)=\OG(\EE_8)$.
See~\eqref{eq:rtimes}.
We then  put
\[
\barW:=W(\EE_8)/\{\pm \id\}, 
\qquad 
\barDelta:=\Delta(\EE_8)/\{\pm \id\}.
\]
Then $\barW$ acts on $\barDelta$ and hence on the set 
$\barDelta\spset{k}$ 
of $k$-element subsets of $\barDelta$.
We define 
\[
N(k):=\textrm{the number of $\barW$-orbits in  $\barDelta^{\{k\}}$.}
\]
\par
Finally, we define the topological types of plane curves as follows.
\begin{definition}\label{def:homeo}
Two plane curves $D$ and $D\sprime$ are said to 
have \emph{the same embedding topology}
if there exists a homeomorphism between 
$(\PP^2, D)$ and $(\PP^2, D\sprime)$.
\end{definition}
Then our main result is stated as follows:
\begin{theorem}\label{thm:main2}
Let $o$ be a general point of $\TTT$.
We consider the plane curves $D_{o, s}$ of combinatorial type $\sigma_k$,
where $s$ runs through $G(C_o)\spset{k}$.
Then, 
classifying these curves by the embedding topology yields  exactly $N(k)$ classes.
\end{theorem}
Since $|\barDelta|=120$ and $|\barW|=348364800$,
we have the inequality~\eqref{eq:Nbinom}.
Therefore Theorem~\ref{thm:main2} implies Theorem~\ref{thm:main1}.
%
%
%%%%%%%%%%%%%%%%%
%
%
\section{Del Pezzo surfaces}\label{sec:monodromy}
In this section, we investigate del Pezzo surfaces.
For the classical results about del Pezzo surfaces, 
we refer the reader to~\cite{Demazure1980, Manin1986}. % and~\cite[Chapter III-3]{Kollar1996}.
In Section~\ref{subsec:Pic}, 
we study the Picard lattice of a  del Pezzo surface.
In Section~\ref{subsec:monodromy},
we present a simple method (Corollary~\ref{cor:index})
for studying   the monodromy  
of a family of del Pezzo surfaces.
In Sections~\ref{subsec:cubic}--\ref{subsec:bianticanonical},
we apply this method 
to natural families of del Pezzo surfaces of degree $3$, $2$, and $1$, respectively. 
\subsection{Picard lattice}\label{subsec:Pic}
Let $d$ be a positive integer $\le 6$. 
We put 
\[
n:=9-d.
\]
Let $X$ be a del Pezzo surface of degree $d$,
that is, a smooth surface 
whose anti-canonical class 
$\alpha_{X}:=[-K_X]$ 
is ample of self-intersection number $d$.
The Picard lattice $\Pic(X)$ of $X$ is of rank $n+1$,
and is canonically isomorphic to $H^2(X)$.  
There exists a birational morphism 
\begin{equation*}\label{eq:beta}
\beta\colon X\to {\bfP}^{2}
\end{equation*}
that is a blowing-up at  distinct $n$ points on ${\bfP}^{2}$,
and $\Pic(X)$ has a basis $h, e_1, \dots, e_n$,
where $h$ is the class of the pullback of a line  on ${\bfP}^{2}$, and $e_1, \dots, e_n$
are the classes of exceptional curves.
With respect to this basis, 
the Gram matrix of $\Pic(X)$  is the diagonal matrix
with diagonal entries $1, -1, \dots, -1$, 
and the anti-canonical   class $\alpha_{X}\in \Pic(X)$
is written as 
\begin{equation*}\label{eq:aX}
\alpha_{X}=(3, -1, \dots, -1).
\end{equation*}
\begin{table}
\[
% Table1.tex
% Made by Read("Task20250409a.txt");
%
\begin{array}{cccrccc}
d & n & \tau_n & |W(R(X))| &  |\Aut(\tau_n)| & |L(X)| & |L^{[n]}(X)\sprime| \\ 
\hline
6 & 3 & A_{1}+A_{2} & {2}^{2} \cdot {3} & 2 & 6 & 2\mystruth{14pt} \\ 
5 & 4 & A_{4} & {2}^{3} \cdot {3} \cdot {5} & 2 & 10 & 5 \\ 
4 & 5 & D_{5} & {2}^{7} \cdot {3} \cdot {5} & 2 & 16 & 16 \\ 
3 & 6 & E_{6} & {2}^{7} \cdot {3}^{4} \cdot {5} & 2 & 27 & 72 \\ 
2 & 7 & E_{7} & {2}^{10} \cdot {3}^{4} \cdot {5} \cdot {7} & 1 & 56 & 576 \\ 
1 & 8 & E_{8} & {2}^{14} \cdot {3}^{5} \cdot {5}^{2} \cdot {7} & 1 & 240 & 17280 \\ 
\end{array}
\]
\caption{Lattice theoretic data}\label{table:Table1}
\end{table}%
The orthogonal complement 
\[
R(X):=(\alpha_{X})\sperp
\]
of $\alpha_{X}$ in $\Pic(X)$ 
is a negative-definite root lattice of type $\tau_n$,
where $\tau_n$ is given in Table~\ref{table:Table1}.
Let $W(R(X))$ denote the Weyl group of the lattice $R(X)$.
The order of $W(R(X))$ is obtained  from the $\ADE$-type $\tau_n$.
%See, for example,~\cite[Section~2.11]{Humphreys1990}.
The root lattice $R(X)$
has a basis $r_1, \dots, r_n$ 
consisting of $(-2)$-vectors 
whose dual graph is the ordinary Dynkin diagram of type $\tau_n$.
We have 
\begin{equation}\label{eq:rtimes}
\OG(R(X))=W(R(X))\rtimes \Aut(\tau_n),
\end{equation}
where $\Aut(\tau_n)$ is  the group of symmetries  of 
the root system  
$\{r_1, \dots, r_n\}$. % in $R(X)$.
We put
\[
\OG(\Pic(X), \alpha_{X}):=\set{g\in \OG(\Pic(X))}{\alpha_{X}^g=\alpha_{X}}.
\]
Then we have a natural  homomorphism
\begin{equation}\label{eq:restriction}
\OG(\Pic(X), \alpha_{X}) \to \OG(R(X))
\end{equation}
given by the restriction $g\mapsto g|R(X)$.
It is obvious that~\eqref{eq:restriction} is injective.
\begin{proposition}\label{prop:imageisW}
The image of the homomorphism~\eqref{eq:restriction} is equal to $W(R(X))$.
\end{proposition}
\begin{proof}
We put $A:=\ZZ\alpha_{X}$, $R:=R(X)$, and consider 
their discriminant groups 
\[
d_A:=A\dual/A, \quad d_R:=R\dual/R,
\]
where $A\dual$ and $R\dual$ are 
the dual lattices of $A$ and $R$, respectively.
%Both $d_A$ and $d_R$ are isomorphic to $\ZZ/d\ZZ$.
%Note that $A$ and $R$ are primitive in $\Pic(X)$.
The unimodular lattice $\Pic(X)$,
which is a submodule of $A\dual \oplus R\dual$
containing $A\oplus R$,
gives rise to the graph 
\[
\Pic(X)/(A\oplus R)\;\subset\; d_A\times d_R
\] 
of an isomorphism $d_A\cong d_R$.
%See,~for example,~\cite{theNikulin}.
Hence 
an isometry $g\in \OG(R)$ extends to 
an isometry 
of $\Pic(X)$ that acts on $A$ trivially
if and only if 
$g$ acts on $d_R$ trivially.
It is easy to see that a reflection with respect to a $(-2)$-vector acts trivially on the discriminant group.
On the other hand, 
a direct calculation shows that a  non-trivial element of $\Aut(\tau_n)$ (if there exists any) acts
 non-trivially on $d_R$.
 Hence Proposition~\ref{prop:imageisW} follows from~\eqref{eq:rtimes}.
\end{proof}
A smooth rational curve $l$ on $X$ is called a \emph{line} if $\intf{l, \alpha_{X}}=1$.
Every line has a self-intersection number $-1$, and 
the set of lines in $X$ is identified with 
\[
L(X):=\set{\lambda\in \Pic(X)}{\intf{\lambda, \lambda}=-1,\;\intf{\lambda, \alpha_{X}}=1}.
\]
By abuse of language,
we sometimes call an element of $L(X)$ a line.
We can enumerate all the elements of $L(X)$ explicitly.
Let  $L^{[n]}(X)$ denote the set of all ordered $n$-tuples 
\[
\blambda=[\lambda_1, \dots, \lambda_n]
\]
of lines such that $\intf{\lambda_i, \lambda_j}=0$ for any $i, j$ with  $i\ne j$.
%Note that we have 
%\[
%\bm{e}:=[e_1, \dots, e_n]\in L^{[n]}(X),
%\]
%where $e_1, \dots, e_n$ are the classes of the exceptional curves of~\eqref{eq:beta}.
By Proposition~\ref{prop:imageisW} and~\cite[Proposition II-4]{Demazure1980},  we have the following:
\begin{corollary}\label{cor:freeandtransitive}
The natural action of $\OG(\Pic(X), \alpha_{X})\cong W(R(X))$ 
on $L^{[n]}(X)$ is free and transitive.
\qed
\end{corollary}
For $\blambda=[\lambda_1, \dots, \lambda_n]\in L^{[n]}(X)$,
we have a birational morphism
to a projective plane 
that is the contraction of the lines 
$l_1, \dots, l_n$ whose classes are $\lambda_1, \dots, \lambda_n$.
We denote this blowing-down morphism by
\begin{equation}\label{eq:betalambda}
\beta_{\blambda}\;\colon\; X \;\to\; \bP(X/\blambda).
\end{equation}
\subsection{Monodromy}\label{subsec:monodromy}
Let $f\colon \XXX \to \UUU$ be a family of del Pezzo surfaces of degree  $d$.
We assume that the parameter space $\UUU$ is smooth and irreducible.
For a point $u$ of $\UUU$, we denote by $X_u$ the fiber $f\inv(u)$.
Instead of $\alpha_{X_u}$, we denote by 
$\alpha_u\in \Pic(X_u)$ 
the anti-canonical class of $X_u$.
We choose a base point $b\in \UUU$.
The local system 
\[
R^2f_* \ZZ\; \to\; \UUU
\]
is a family of the lattices $\Pic(X_u)\cong H^2(X_u)$,
and it has a section $u\mapsto \alpha_{u}$.
Hence we obtain a monodromy homomorphism 
\begin{equation}\label{eq:mon}
\mon\;\colon\; \pione(\UUU, b)\;\to\; \OG(\Pic(X_b), \alpha_b)\cong W(R(X_b)).
\end{equation}
We investigate  the surjectivity of the monodromy homomorphism $\mon$. %~\eqref{eq:mon}.
\par
We consider the \'etale covering
\[
\LLL^{[n]}\to\UUU
\]
whose fiber over $u\in\UUU$ is the set $L^{[n]}(X_u)$ of ordered sets of disjoint $n$ lines in $X_u$.
The associated monodromy action 
of $\pione(\UUU, b)$ on $L^{[n]}(X_b)$ 
is given by
\[
[\gamma]\;\;\mapsto\;\;(\blambda\;\mapsto \; \blambda^{ \mon([\gamma])})
\]
for $[\gamma]\in \pione(\UUU, b)$,
where $\blambda\mapsto \blambda^{ \mon([\gamma])}$  
denotes the action of the element $ \mon([\gamma])$ of 
$ \OG(\Pic(X_b),  \alpha_b)$ on $L^{[n]}(X_b)$.
The orbits of the monodromy action on $L^{[n]}(X_b)$ 
correspond bijectively to the connected components of $\LLL^{[n]}$.
Hence, by Corollary~\ref{cor:freeandtransitive}, we obtain the following:
\begin{corollary}\label{cor:index}
The index of the image %$\Gamma:=\Image \mon$ 
of  the monodromy homomorphism~$\mon$ in $W(R(X_b))$ is 
equal to the number of the connected components of $\LLL^{[n]}$.
\qed
\end{corollary}
We fix a projective plane ${\bfP}^{2}$.
For an ordered set  $\bp= [p_1, \dots, p_n]\in ({\bfP}^{2})^n$ of distinct $n$ points of ${\bfP}^{2}$, 
we denote by 
\[
\beta_{\bp}\;\;\colon\;\;  Y_{\bp}\;\to\; {\bfP}^{2}
\]
 the blowing-up at 
the points $p_1, \dots, p_n$.
\begin{definition}\label{def:PPPn}
Let $\PPP_n$ be the Zariski open subset of $({\bfP}^{2})^n$
consisting of all ordered sets  $\bp=[p_1, \dots, p_n]$ of distinct $n$ points of ${\bfP}^{2}$
such that 
\begin{enumerate}[label={\rm{(\roman *)}}]
\item  no three points in  $\bp$ are on a line,
\item  no six points in  $\bp$ are on a conic, and 
\item   there exists no cubic curve passing through $7$  points in  $\bp$
and having a double point at the $8$th point.
\end{enumerate}
\end{definition}
\begin{theorem}[Th\'eor\`eme~II-1 in~\cite{Demazure1980}]
\label{thm:PPPn}
The surface  $Y_{\bp}$ is a del Pezzo surface of degree $d=9-n$ if and only if $\bp\in \PPP_n$.
\qed
\end{theorem}
For $\bp= [p_1, \dots, p_n] \in \PPP_n$, we have a distinguished element 
\[
\brho_{\bp}=[\rho_1, \dots, \rho_n]
\]
of $L^{[n]}(Y_{\bp})$, 
where $\rho_i$ is the class of the exceptional curve over $p_i$.
We consider the incidence  variety
\[
\III:=\bigset{(u,\bp, g)}{%
\parbox{7.7cm }{$u \in \UUU, \, \bp\in \PPP_n$, 
and $g$ is an isomorphism $X_u\isom Y_{\bp}$
}%
}
\]
with the projections
\[
\UUU\;\mapleftsp{\varpi_{1}}\;\III\;\maprightsp{\varpi_2}\; \PPP_n.
\]
A point of $\LLL^{[n]}$ is written as $(u, \blambda)$,
where $u\in \UUU$ and $\blambda\in L^{[n]}(X_u)$.
Then we can lift $\varpi_1\colon \III\to \UUU$  to 
$\varpi_1^{\LLL}\colon \III\to \LLL^{[n]}$
by setting
\[
\varpi_1^{\LLL}(u, \bp, g):=(u, \blambda),
\]
where $\blambda$ is the element of 
$L^{[n]}(X_u)$ that is mapped to $\brho_{\bp}\in L^{[n]}(Y_{\bp})$ by 
the bijection $L^{[n]}(X_u) \isom L^{[n]}(Y_{\bp}) $ 
induced by  $g\colon X_u\isom Y_{\bp}$:
\[
\begin{array}{ccccc}
\LLL^{[n]} & \mapleftsp{\varpi_1^{\LLL}} & \III &  \maprightsp{\varpi_2} & \PPP_n\mystruth{15pt} \\
\mapdownleft{} &\rlap{$\swarrow$\scriptsize$\varpi_1$} &&& \\
\,\UUU\rlap{\;.}&&&&
\end{array}
\]
The fiber of $\varpi_1^{\LLL}$ over $(u, \blambda)\in \LLL^{[n]}$ is the variety of all isomorphisms
\[
\bP(X_u/\blambda)\;\isom\;{\bfP}^{2},
\]
where $\bP(X_u/\blambda)$ is defined in~\eqref{eq:betalambda}.
Indeed, if we have $\varpi_1^{\LLL}(u,\bp, g)=(u, \blambda)$, then the isomorphism $g\colon X_u\isom Y_{\bp}$ maps  
the exceptional curves of the blowing-down $\beta_{\blambda}\colon X_u\to  \bP(X_u/\blambda)$ to
the exceptional curves of the blowing-down $\beta_{\bp}\colon Y_{\bp}\to {\bfP}^{2}$, 
and hence 
$g$ induces an isomorphism $\bP(X_u/\blambda)\isom{\bfP}^{2}$.
Conversely, if we are given a point $(u, \blambda)$ of $\LLL^{[n]}$ and an isomorphism $\bar{g}\colon \bP(X_u/\blambda)\isom{\bfP}^{2}$,
then, setting $\bp$ to be the image of the centers of 
$\beta_{\blambda}$ by $\bar{g}$ 
and lifting $\bar{g}$
to the isomorphism $g\colon X_u\isom Y_{\bp}$ between  their  blow-ups, 
we obtain  a point $(u,\bp, g)$  in the fiber of $\varpi_1^{\LLL}$ over $(u, \blambda)$.
\par
The variety of isomorphisms %between the projective planes $\bP(X_u/\blambda)$ and ${\bfP}^{2}$
$\bP(X_u/\blambda)\cong {\bfP}^{2}$
is isomorphic to $\PGL(3, \CC)$.
Hence $\varpi_1^{\LLL}$ gives a bijection from the set of connected components of $\III$
to that of $\LLL^{[n]}$.
We investigate the connected components  of $\III$ by 
the second projection $\varpi_2\colon \III \to \PPP_n$.
%
%
%\subsection{The case $d=3$}
\subsection{Family of cubic surfaces}\label{subsec:cubic}
The anti-canonical model of a del Pezzo surface of degree $d=3$
is a smooth cubic surface.
%Conversely, every smooth cubic surface 
%is the anti-canonical model of a del Pezzo surface of degree $3$.
%\par
We fix a projective space $\PP^3$,
and consider the family $\XXX\to \UUU$ of smooth cubic surfaces, where 
$\UUU$ is  the Zariski open subset of 
$|\OOO_{\PP^3} (3)|\cong \PP^{19}$
parameterizing all smooth cubic surfaces.
\par
The following reproduces the result of Harris~\cite{Harris1979}  on the Galois group of $27$ lines in a smooth cubic surface.
\begin{proposition}\label{prop:WE6}
For the  family 
 $\XXX\to \UUU$ 
 of smooth cubic surfaces,
 the monodromy homomorphism $\mon$  %in~\eqref{eq:mon}
 is surjective onto the Weyl group $W(R(X_b))$ 
 of type $E_6$.
\end{proposition}
\begin{proof}
For any point $\bp\in \PPP_6$,
the fiber of $\varpi_2$ over $\bp$ is 
the variety of all isomorphisms between $\PP^3$
and the projective space 
\[
|\alpha_{\bp}|\dual=\PP^{*}(H^0(Y_{\bp}, \OOO(\alpha_{\bp}))),
\]
where $\alpha_{\bp}$ is the anti-canonical class of $Y_{\bp}$.
This variety is isomorphic to $\PGL(4, \CC)$.
Hence $\LLn{6}$ is connected.
Therefore $\mon$ is surjective by Corollary~\ref{cor:index}.
\end{proof}
%
%
%%%%%%%%%%%%%%%%%Changed 2026/Apr/02
See~\cite{Shimada2026} for an application of Proposition~\ref{prop:WE6}.
%%%%%%%%%%%%%%%%%
%

%\subsection{The case $d=2$}
\subsection{Family of quartic double planes}\label{subsec:quartic}
The anti-canonical model of a del Pezzo surface
$X$ of degree $d=2$ is a double plane
$X\to \PP^2$ branching along a smooth quartic curve.
%Conversely, every double plane 
%branching along a smooth quartic curve
%is the anti-canonical model of a del Pezzo surface of degree $2$.
%%
%\par
%
We fix a projective plane $\PP^2$. 
Let  $\UUU$ be the Zariski open subset of 
$|\OOO_{\PP^2} (4)|\cong \PP^{14}$
that parameterizes all smooth quartic curves in $\PP^2$.
For each $u\in \UUU$,
we denote by  $B_u\subset \PP^2$  the corresponding quartic curve.
We consider the family 
 $\XXX\to \UUU$ 
 of smooth quartic double planes,
that is, 
 $\XXX$ is the double cover of $\PP^2\times \UUU$ with the projection  $ \XXX\to \UUU$ 
 whose fiber  $X_u$ over $u\in \UUU$
 is  the double plane $X_u\to \PP^2$
branching along $B_u$.
\par
The following proposition was stated in~\cite{Shimada2022}, but the proof was incomplete.
A weaker result concerning
the Galois group of the $28$ bitangents of a smooth quartic curve had been proved  in  Harris~\cite{Harris1979}.
\begin{proposition}\label{prop:WE7}
For the  family 
 $ \XXX\to \UUU$ 
 of smooth quartic double planes,
 the monodromy homomorphism $\mon$ %in~\eqref{eq:mon}
 is surjective onto the Weyl group $W(R(X_b))$ 
 of type~$E_7$.
\end{proposition}
\begin{proof}
Let $\bp$ be a point of $\PPP_7$,
and let $Y_{\bp} \to {\sfP}_{\bp}^2$ 
be the anti-canonical model of $Y_{\bp}$.
Let $\sfB_{\bp}\subset {\sfP}_{\bp}^2$ be the branch curve of $Y_{\bp} \to {\sfP}_{\bp}^2$.
The fiber of $\varpi_2\colon \III\to \PPP_7$ 
over $\bp$ consists 
of pairs $(u, g)$,
where $u$ is a point of $\UUU$
and $g$ is an isomorphism from $X_u$ to $Y_{\bp}$.
An isomorphism $g$ from $X_u$ to $Y_{\bp}$
induces an isomorphism
\[
\bar{g}\;\colon \; \PP^2 \isom  {\sfP}_{\bp}^2.
\]
Conversely,
suppose that an isomorphism $\gamma\colon \PP^2\isom {\sfP}_{\bp}^2$ is given.
Let $u\in \UUU$ be the point such that $B_{u}=\gamma\inv(\sfB_{\bp})$.
Then $\gamma$ admits exactly \emph{two}  liftings
\[
g_1\;\colon \;  X_u \isom Y_{\bp},
\quad
g_2\;\colon \;  X_u \isom Y_{\bp},
\]
which differ by the deck-transformation of $X_u$ over  $\PP^2$.
Since $\PGL(3, \CC)$ is smooth and irreducible,
we see that the fiber of $\varpi_2\colon \III\to \PPP_7$ 
over $\bp$ has at most two connected components.
Therefore
$\III$  has at most two connected components,
and so does $\LLn{7}$.
Consequently,
the index $[W: \Gamma]$ of the image
\[
\Gamma:=\Image(\mon)
\]  
of $\mon$ in $W:=W(R(X_b))$  is at most $2$.
\par
We assume 
\begin{equation}\label{eq:absurdum}
[W: \Gamma]=2, 
\end{equation}
and derive a contradiction.
Note  that the Weyl group $W$ of type $E_7$ contains a simple group $G$  
as the kernel of $\det \colon W \to \{\pm 1\}$.
See~\cite[page 46]{ATLAS1985}.
We show in the next paragraph  that  $\Gamma$  contains an element of determinant $-1$.
By assumption~\eqref{eq:absurdum}, we see that 
$G\cap \Gamma$ is a normal subgroup of $G$ with index $2$,
which contradicts the simplicity of $G$.
\par
Let $\HHH\subset |\OOO_{\PP^2}(4)|\cong \PP^{14}$ 
be the hypersurface that parameterizes singular quartic curves.
Then $\HHH$ is irreducible. 
Let $q$ be a general point of $\HHH$.
Then the corresponding quartic curve $B_q\subset \PP^2$ 
has an  ordinary node as its only singularity, 
and hence the double plane $X_q\to \PP^2$ branching along $B_q$ has 
an ordinary double point as its only singularity.
We choose a sufficiently small closed disc $D\subset |\OOO_{\PP^2}(4)|$ 
intersecting $\HHH$ at $q$ transversely,
and let
$\gamma\colon[0,1]\to \UUU$ be a loop 
that goes from the base point $b$ to a point $q\sprime\in \partial D$ along a path $\tau$,
makes a round trip along $\partial D$ in  positive-direction,
and retraces  the  path $\tau$ backwards to $b$.
The monodromy action on $H^2(X_b)$ by   $[\gamma]\in \pione(\UUU, b)$ is calculated by
the \emph{local Picard--Lefschetz formula}.
(See, for example, \cite{Lamotke1981}.)
We have  a \emph{vanishing cycle} $v\in H^2(X_b)$
corresponding to the ordinary double point of $X_q$,
which satisfies $\intf{ \alpha_b, v} =0$ and $\intf{v, v} =-2$,
and the monodromy  on $H^2(X_b)\cong \Pic(X_b)$ by $[\gamma]$ is the reflection $x\mapsto x+\intf{v, x}v$
with respect to $v$.
Hence we have $\det(\mon([\gamma]))=-1$. 
Note that the identification $\OG(\Pic(X_b), \alpha_b)\cong W(R(X_b))$ 
preserves the determinant.
Therefore $\Gamma$  contains an element of determinant $-1$.
\end{proof}
\subsection{Family of bi-anti-canonical models of del Pezzo surfaces of degree
\texorpdfstring{$1$}{one}}\label{subsec:bianticanonical}
Let $X$ be a del Pezzo surface of degree $d=1$.
Then 
the complete linear system $|2\alpha_X| $
of bi-anti-canonical divisors of $X$ gives rise to a double covering
\[
X\to Q\subset\PP^3
\]
of a singular quadric surface $Q$ of rank $3$ (a quadric cone)
that branches along $B\cup\{V\}$,
where $B$ is a smooth member of $|\OOO_{Q}(3)|$ and $V\in Q$ is the vertex.
Conversely, 
for a quadric cone $Q$ with the vertex $V\in Q$ and 
a smooth member $B$ of  $|\OOO_{Q}(3)|$,
the double cover $X\to Q$ branching  along $B\cup\{V\}$
is the bi-anti-canonical model of a del Pezzo surface $X$ of degree $1$.
\par
We fix 
a quadric cone $Q$ with the vertex $V\in Q$.
Let $\UUU$ denote the Zariski open subset of 
$|\OOO_{Q} (3)|\cong \PP^{15}$
that parameterizes all smooth members.
For $u\in \UUU$,
we denote by  $B_u\subset Q$  the corresponding curve.
We consider the family 
 $\XXX\to \UUU$ 
  of del Pezzo surfaces 
of degree $1$
such that
 $\XXX$ is the double cover of $Q\times \UUU$ with  the  projection  $\XXX\to \UUU$ 
 whose fiber  $X_u$ over $u\in \UUU$
 is the double cover $X_u\to Q$
branching along $B_u\cup \{V\}$.
\begin{proposition}\label{prop:WE8}
For the  family 
 $ \XXX\to \UUU$ 
 of bi-anti-canonical models of del Pezzo surfaces of degree $1$,
 the monodromy homomorphism $\mon$
 is surjective onto the Weyl group $W(R(X_b))$ 
 of type $E_8$.
\end{proposition}
\begin{proof}
The first half of the proof is almost the same as that of Proposition~\ref{prop:WE7}.
Let $\bp$ be a point of $\PPP_8$, 
and let $Y_{\bp} \to {\sfQ}_{\bp}$ 
be the bi-anti-canonical model of $Y_{\bp}$.
Let $\sfB_{\bp}\subset {\sfQ}_{\bp}$ be the curve component of the branch locus  of $Y_{\bp} \to {\sfQ}_{\bp}$.
The fiber of $\varpi_2\colon \III\to \PPP_8$ 
over $\bp$ consists 
of pairs $(u, g)$,
where $u$ is a point of $\UUU$
and $g$ is an isomorphism from $X_u$ to $Y_{\bp}$.
An isomorphism $g$ from $X_u$ to $Y_{\bp}$
induces an isomorphism $\bar{g}\colon Q \isom {\sfQ}_{\bp}$.
Conversely,
suppose that an isomorphism 
$\gamma\colon Q \isom {\sfQ}_{\bp}$ is given.
Let $u\in \UUU$ be the point such that $B_{u}=\gamma\inv(\sfB_{\bp})$.
Then $\gamma$ lifts to exactly \emph{two} isomorphisms 
from $X_u$ to $Y_{\bp}$.
Since the variety 
of isomorphisms from $Q $ to ${\sfQ}_{\bp}$
is smooth and irreducible, 
 the fiber of $\varpi_2$ 
over $\bp$ has at most two connected components.
Therefore  $\III$  has at most two connected components,
and hence 
the index $[W: \Gamma]$ of the image
$\Gamma:=\Image(\mon)$ of the monodromy $\mon$ 
in $W:=W(R(X_b))$  is at most $2$.
We assume 
$[W: \Gamma]=2$, 
and derive a contradiction.
\par
Note  that the Weyl group $W$ of type $E_8$ has the  structure $2.G.2$,
where 
\[
2.G=\Ker\; (\det \colon W\to \{\pm 1\}),
\quad 
G.2=W/\{\pm \id\}, 
\]
and $G$ is a simple group. 
See~\cite[page 85]{ATLAS1985}.
By the assumption $[W: \Gamma]=2$,
we have $|\Gamma|=|2.G|=|G.2|$.
Let $\Delta$ be the set of $(-2)$-vectors 
in the root lattice $R(X_b)$. 
For $r\in \Delta$,
let $s_r\in W$ denote the reflection with respect to $r$.
\par
There exists a member $B_q$  of $|\OOO_{Q}(3)|$
that does not pass through $V$ and  has an ordinary node as its only singularity.
By the argument using the local Picard--Lefschetz formula
as in the proof of Proposition~\ref{prop:WE7},
we see that there exists a vanishing cycle
$v\in \Delta$ corresponding to the ordinary double point of $X_q$
such that $\Gamma$ contains the reflection $s_v$.
Since $\det s_v=-1$,
we see that $\Gamma\cap (2.G)$ is of index $2$ in $2.G$.
If $-\id \in \Gamma$, then $-\id\in \Gamma\cap (2.G)$ 
and hence $(\Gamma\cap (2.G))/\{\pm\id\}$ would be a  subgroup of $(2.G)/\{\pm\id\}=G$ with index $2$.
Since $G$ is simple, we have  $-\id \notin \Gamma$.
Hence $\Gamma$ is mapped isomorphically  onto $G.2=W/\{\pm \id\}$,
which acts on $\Delta/\{\pm \id\}$ transitively.
Since $s_v\in \Gamma$ and $g\inv\cdot s_v\cdot g=s_{v^g}=s_{-v^g}$, 
we see that $s_r\in \Gamma$  for any $r\in \Delta$.
Thus we obtain $\Gamma=W$,
which is a contradiction.
\end{proof}
  \section{Lines in a del Pezzo surface of degree one}\label{sec:lines}
We choose a \emph{general}  point $b$ of the parameter space $\UUU$
of the family 
$\XXX\to \UUU$ of bi-anti-canonical models of del Pezzo surfaces of degree $1$ 
treated in Section~\ref{subsec:bianticanonical}.
The purpose of this section is to 
investigate  
the configuration of lines  in the del Pezzo surface  $X_b$.
\par
In Section~\ref{subsec:Bertini}, 
we introduce the notions of 
\emph{Bertini involution} and \emph{tangent plane sections}.
In Section~\ref{subsec:orbitdecomp}, using Proposition~\ref{prop:WE8}, 
we describe the orbit decompositions
of the set of lines in $X_b$ by the monodromy action.
In Section~\ref{subsec:linesinbP2}, we describe the lines in $X_b$ as plane curves on 
the projective plane ${\bfP}^{2}$ obtained by contracting $8$ disjoint lines in $X_b$.
In Section~\ref{subsec:general}, 
we confirm that 
the union of lines in $X_b$ has only ordinary double points as its singularities.
\subsection{Bertini involution}\label{subsec:Bertini}
We have an orthogonal direct-sum decomposition 
\[
\Pic(X_b)=\ZZ \alpha_b \oplus R(X_b),
\]
where $\alpha_b$ is the anti-canonical class of $X_b$.
The orthogonal projection from $\Pic(X_b)$ to $R(X_b)$ induces a bijection 
\begin{equation}\label{eq:LX}
L(X_b)\cong  \Delta(R(X_b))
\end{equation}
between  the set $L(X_b)$ of lines in $X_b$ and  the set $\Delta(R(X_b))$
of $(-2)$-vectors of the root lattice $R(X_b)$ of type $E_8$.
For a line $l\in L(X_b)$, we denote by 
\[
[l]_R:=[l]-\alpha_b \;\in\; \Delta(R(X_b))
\]
the corresponding $(-2)$-vector. 
Let 
\[
\varphi\colon X_b\to Q
\] 
be the bi-anti-canonical model of $X_b$,
where $Q\subset \PP^3$ is a   quadric cone with the vertex $V\in Q$,
and let $B_b\subset Q$ be the curve component of  the branch locus of 
$\varphi$.
Then  $B_b$ is a smooth member of $|\OOO_{Q}(3)|$. 
The deck-transformation of  $\varphi$
is called the \emph{Bertini involution},  and is denoted by
\[
\involB\colon X_b\isom X_b.
\]
We call a pair $\{l, l\sprime\}$ of lines in $X_b$ an \emph{$\involB$-pair}
if $l\sprime=\involB(l)$,
and say that $l\sprime=\involB(l)$ is the \emph{$\involB$-partner of $l$}.
It is easy to see that,
for lines $l, l\sprime$ in $X_b$,  the following are equivalent:
(i) the pair $\{l, l\sprime\}$ is an $\involB$-pair, 
(ii)~$\intf{l, l\sprime}=3$, 
(iii)~$[l]+[l\sprime]=2\alpha_b$, and
(iv)~$[l]_R+[l\sprime]_R=0$.
\begin{definition}\label{def:tps}
A plane section $H\cap Q$ of $Q$,
where $H$ is a linear plane in $\PP^3$,
is called a \emph{tangent plane section for $B_b$}
if $H$ does not pass through the vertex $V$ 
% We need this condition.
%Consider the case where $H$ is a double of a ruling.
and the local intersection number at each intersection point 
of $H$ and $B_b$ is even.
We put
\[
S(B_b):=\textrm{the set of 
tangent plane sections  for $B_b$}.
\]
\end{definition}
The image $\varphi(l)$ of a line $l\subset X_b$
by $\varphi$ is a  tangent plane section for $B_b$.
Conversely, the pullback by $\varphi$ of a tangent plane section 
is the union of a line and its $\involB$-partner.
Hence we have natural identifications 
\begin{equation}\label{eq:barLX}
\barL(X_b):=L(X_b)/\angs{\involB}\;\;\cong\;\;  
\barDelta(R(X_b)):=\Delta(R(X_b))/\{\pm \id\} \;\;\cong\;\; S(B_b).
\end{equation}
In particular,
there exist
exactly $120$ tangent plane sections for $B_b$. 
\begin{remark}\label{rem:thetacharacteristics}
The smooth $(2,3)$-complete intersection  $B_b\subset \PP^3$
is the canonical model of a genus $4$ curve with a vanishing theta constant, and 
the tangent plane sections for $B_b$ are in one-to-one correspondence 
with the odd theta-characteristics of $B_b$.
See Chapter IV and Appendix~B of~\cite{ACGH1985}.
\end{remark}
%
%%%%%%%%%%%%
%
%
\subsection{Orbit decomposition by the monodromy}\label{subsec:orbitdecomp}
By Proposition~\ref{prop:WE8},  we can compute 
the monodromy actions of $\pione(\UUU, b)$
on the sets in~\eqref{eq:LX} and~\eqref{eq:barLX}
explicitly.
We describe the orbit decompositions of 
$L(X_b)^{\{k\}}$ and $\barL(X_b)^{\{k\}}$ 
under these monodromy actions.
%
%\par
%\medskip
\subsubsection{The action on  $L(X_b)^{\{k\}}$}
The monodromy action on the set $L(X_b)^{\{k\}}$ of $k$-element subsets of $L(X_b)$ 
for small $k$ is as follows.
The numbers of orbits are given as follows:
%
%gap> Read("Task250703c.txt");
%1 1 
%2 4 
%3 12 
%4 62 
%5 378 
%6 3557 
%7 45282 
%
\begin{equation*}%\label{eq:Nks}
\begin{array}{c|ccccccc}
k & 1 & 2 & 3 & 4 & 5 & 6 &7  \\
\hline 
 & 1& 4 & 12 & 62 & 378 &  3557 & 45282\rlap{\;.}
\end{array}
\end{equation*}
%
%gap> Read("Task250610d.txt");
%1 transitive 
%1  done ________________ 
%[ 0, 6720 ] 
%[ 1, 15120 ] 
%[ 2, 6720 ] 
%[ 3, 120 ] 
%2  done ________________ 
%[ [ 0, 0, 0 ], 60480 ] 
%[ [ 0, 0, 1 ], 181440 ] 
%[ [ 0, 0, 2 ], 6720 ] 
%[ [ 0, 1, 1 ], 483840 ] 
%[ [ 0, 1, 2 ], 362880 ] 
%[ [ 0, 2, 2 ], 181440 ] 
%[ [ 0, 2, 3 ], 13440 ] 
%[ [ 1, 1, 1 ], 302400 ] 
%[ [ 1, 1, 2 ], 483840 ] 
%[ [ 1, 1, 3 ], 15120 ] 
%[ [ 1, 2, 2 ], 181440 ] 
%[ [ 2, 2, 2 ], 2240 ] 
%3  done ________________ 
%
%
\begin{itemize}[itemsep=5pt]
\item The action on $L(X_b)^{\{1\}}=L(X_b)$  is transitive.
\item
The  action on $L(X_b)^{\{2\}}$  has four orbits.
For an orbit $o\subset L(X_b)^{\{2\}}$,
let $m(o)$ denote the intersection number $\intf{l_1, l_2}$ of lines,   
where  $\{l_1, l_2\}\in o$.
Then the four orbits are distinguished  by $m(o)$ as follows:
\[
\begin{array}{c|cccc}
m(o) & 0 & 1 & 2 & 3 \\
\hline 
|o| & 6720 & 15120 & 6720 & 120\mystruth{10pt}
\end{array}.
\]
\item
The  action on $L(X_b)^{\{3\}}$  has $12$ orbits.
For an orbit $o\subset L(X_b)^{\{3\}}$,
let $t(o)$ denote 
the non-decreasing sequence of the intersection numbers $\intf{l_i, l_j}$ 
for $1\le i<j\le 3$,
where  $\{l_1, l_2, l_3\}\in o$.
Then the $12$ orbits are described as follows:
\begin{equation*}\label{eq:ots}
\begin{array}{ll}
t(o) & |o| \\
\hline 
 {[} 0, 0, 0 {]}& 60480\\
 {[} 0, 0, 1 {]}& 181440\\ 
 {[} 0, 0, 2 {]}& 6720\\
 {[} 0, 1, 1 {]}& 483840\\
 {[} 0, 1, 2 {]}& 362880\\
 {[} 0, 2, 2 {]},& 181440\\

\end{array}
\qquad
\begin{array}{ll}
t(o) & |o| \\
\hline
 {[} 0, 2, 3 {]}&13440 \\
{[} 1, 1, 1 {]}& 302400 \\
{[} 1, 1, 2 {]}& 483840 \\
{[} 1, 1, 3 {]}&15120 \\
{[} 1, 2, 2 {]}& 181440 \\
{[} 2, 2, 2 {]}& 2240 \;\;\llap{.}\\
%&\\
%&\\
\end{array}
\end{equation*}
\end{itemize}
Let $\LLL\spset{k}\to \UUU$ be the \'etale covering whose fiber over $u\in \UUU$
is $L(X_u)\spset{k}$.
\begin{corollary}\label{cor:irred}
{\rm (1)}
The space  $\LLL\spset{2}$ consists of  exactly $4$ irreducible components
$\LLL\spset{2}_0, \dots, \LLL\spset{2}_3$, where 
 $\LLL^{\{2\}}_m\to \UUU$ be the family of 
pairs $\{l_1, l_2\}$ of lines in $X_u$ 
such that $\intf{l_1, l_2}=m$.  %\par
{\rm (2)}
The space $\LLL\spset{3}$ consists of exactly $12$ irreducible components
$\LLL\spset{3}_t$, where 
 $t=[t_1, t_2, t_3]$ runs through the list
\begin{equation}\label{eq:tenlist}
\begin{array}{l}
 {[}0,0,0{]},\;\;[0,0,1],\;\;[0,0,2],\;\;[0,1,1],\;\;[0,1,2],\;\; [0,2,2],\;\; [0,2,3], \\
 {[}1,1,1{]},\;\; [1,1,2],\;\; [1,1,3],\;\; [1,2,2],\;\; [2,2,2].
\end{array}
\end{equation}
The \'etale covering  $\LLL^{\{3\}}_t\to \UUU$ is the family of 
triples $\{l_1, l_2, l_3\}$ of lines in $X_u$
such that $[\intf{l_1, l_2}, \intf{l_2, l_3}, \intf{l_1, l_3}]$ is equal to  $t$ up to order.
\qed
\end{corollary}
\subsubsection{The action on  $\barL(X_b)^{\{k\}}$}
The  action on the set  $\barL(X_b)^{\{k\}}$
is identified with 
the action of $\barW$
on $\barDelta^{\{k\}}$ defined in Section~\ref{subsec:Nk}.
Therefore 
the  numbers of orbits are  $N(k)$ and,  for small $k$,
they  are given in
Table~\ref{eq:Nks}.
\begin{itemize}[itemsep=5pt]
\item The  action on the set $\barL(X_b)^{\{1\}}=\barL(X_b)$ of $\involB$-pairs is transitive.
\item
%
%gap> Read("Task250609d.txt");
%[ 1, 2 ] 3360 20 
%[ 1, 15 ] 3780 11 
%_________________ 
%[ 1, 15, 74 ] 37800 [ 11, 11, 11 ] 20-conn comps 6 
%[ 1, 2, 3 ] 30240 [ 20, 20, 20 ] 20-conn comps 1 
%[ 1, 2, 14 ] 90720 [ 20, 11, 20 ] 20-conn comps 2 
%[ 1, 2, 36 ] 120960 [ 20, 11, 11 ] 20-conn comps 4 
%[ 1, 2, 120 ] 1120 [ 20, 20, 20 ] 20-conn comps 2 
%
The  action decomposes  $\barL(X_b)^{\{2\}}$ into two orbits 
of size $3360 $ and  $3780$.
These two orbits are distinguished by 
Figure~\ref{fig:orbitsbarLX2},
where a line is denoted by a circle 
\begin{tikzpicture}[x=1mm, y=1mm]
\draw[fill=white, draw=black,  thick] (0,0) circle (1);
\end{tikzpicture},
\newcommand{\iBpair}{\begin{tikzpicture}[x=1mm, y=1mm]
\draw [line width=2pt] (0,0)--(6,0);
\draw[fill=white, draw=black,  thick] (0,0) circle (1);
\draw[fill=white, draw=black,  thick] (6,0) circle (1);
\end{tikzpicture}}
an $\involB$-pair  is denoted by 
$\iBpair$, 
and the intersection number of distinct two lines is given by 
the number of line-segments connecting the corresponding circles.
\item
The action decomposes  $\barL(X_b)^{\{3\}}$ into five orbits.
These  orbits are depicted in
Figure~\ref{fig:orbitsbarLX3}.
\end{itemize}
%
%%%%%%%%%%%%%%%%%%%
%
\begin{figure}
\begin{tikzpicture}[x=1.1mm, y=1.1mm]
\begin{scope}[shift={(0,0)}]
\draw[double, line width=.5pt, double distance=1.5pt]  (0,0)--(15,0);
\draw[double, line width=.5pt, double distance=1.5pt]  (0,5)--(15,5);
\draw [line width=2pt] (0,0)--(0,5);
\draw [line width=2pt] (15,0)--(15,5);
\draw[fill=white, draw=black,  thick] (0,5) circle (1);
\draw[fill=white, draw=black,  thick] (0,0) circle (1);
\draw[fill=white, draw=black,  thick] (15,5) circle (1);
\draw[fill=white, draw=black,  thick] (15,0) circle (1);
\draw (8,-4) node {size $3360$};
\end{scope}
\begin{scope}[shift={(30,0)}]
\draw[line width=.5pt]  (0,0)--(15,5);
\draw[line width=.5pt]  (0,5)--(15,0);
\draw[line width=.5pt]  (0,0)--(15,0);
\draw[line width=.5pt]  (0,5)--(15,5);
\draw [line width=2pt] (0,0)--(0,5);
\draw [line width=2pt] (15,0)--(15,5);
\draw[fill=white, draw=black,  thick] (0,5) circle (1);
\draw[fill=white, draw=black,  thick] (0,0) circle (1);
\draw[fill=white, draw=black,  thick] (15,5) circle (1);
\draw[fill=white, draw=black,  thick] (15,0) circle (1);
\draw (8,-4) node {size $3780$};
\end{scope}
\end{tikzpicture}
\caption{Orbits in $\barL(X_b)^{\{2\}}$}\label{fig:orbitsbarLX2}
\vskip .7cm
\begin{tikzpicture}[x=1.1mm, y=1.1mm]
\pgfmathsetmacro{\s}{30}
\pgfmathsetmacro{\ss}{\s+\s}
\pgfmathsetmacro{\sss}{\ss+\s}
\pgfmathsetmacro{\ssss}{\sss+\s}

\begin{scope}[shift={(0,0)}]
\coordinate (A1) at  (0,5);
\coordinate (A2) at  (0,10);
\coordinate (B1) at  (4.33,-2.5);
\coordinate (B2) at  (8.66, -5);
\coordinate (C1) at  (-4.33, -2.5);
\coordinate (C2) at  (-8.66, -5);
%%%% 11:11:11
\draw[line width=.5pt] (A1)--(B1);
\draw[line width=.5pt] (A1)--(B2);
\draw[line width=.5pt] (A1)--(C1);
\draw[line width=.5pt] (A1)--(C2);
\draw[line width=.5pt] (A2)--(B1);
\draw[line width=.5pt] (A2)--(B2);
\draw[line width=.5pt] (A2)--(C1);
\draw[line width=.5pt] (A2)--(C2);
\draw[line width=.5pt] (B1)--(C1);
\draw[line width=.5pt] (B1)--(C2);
\draw[line width=.5pt] (B2)--(C1);
\draw[line width=.5pt] (B2)--(C2);
%%%%
\draw[line width=2pt] (A1)--(A2);
\draw[line width=2pt] (B1)--(B2);
\draw[line width=2pt] (C1)--(C2);
\draw[fill=white, draw=black,  thick] (A1) circle (1);
\draw[fill=white, draw=black,  thick] (A2) circle (1);
\draw[fill=white, draw=black,  thick] (B1) circle (1);
\draw[fill=white, draw=black,  thick] (B2) circle (1);
\draw[fill=white, draw=black,  thick] (C1) circle (1);
\draw[fill=white, draw=black,  thick] (C2) circle (1);
\draw (0,-9) node {size $37800$};
\end{scope}
\begin{scope}[shift={(\s,0)}]
\coordinate (A1) at  (0,5);
\coordinate (A2) at  (0,10);
\coordinate (B1) at  (4.33,-2.5);
\coordinate (B2) at  (8.66, -5);
\coordinate (C1) at  (-4.33, -2.5);
\coordinate (C2) at  (-8.66, -5);
%%%% 11:11:20
\draw[line width=.5pt] (A1)--(B1);
\draw[line width=.5pt] (A1)--(B2);
\draw[line width=.5pt] (A1)--(C1);
\draw[line width=.5pt] (A1)--(C2);
\draw[line width=.5pt] (A2)--(B1);
\draw[line width=.5pt] (A2)--(B2);
\draw[line width=.5pt] (A2)--(C1);
\draw[line width=.5pt] (A2)--(C2);
\draw[double, line width=.5pt, double distance=1.5pt]   (B1)--(C1);
\draw[double, line width=.5pt, double distance=1.5pt]   (B2)--(C2);
%%%%
\draw[line width=2pt] (A1)--(A2);
\draw[line width=2pt] (B1)--(B2);
\draw[line width=2pt] (C1)--(C2);
\draw[fill=white, draw=black,  thick] (A1) circle (1);
\draw[fill=white, draw=black,  thick] (A2) circle (1);
\draw[fill=white, draw=black,  thick] (B1) circle (1);
\draw[fill=white, draw=black,  thick] (B2) circle (1);
\draw[fill=white, draw=black,  thick] (C1) circle (1);
\draw[fill=white, draw=black,  thick] (C2) circle (1);
\draw (0,-9) node {size $120960$};
\end{scope}
\begin{scope}[shift={(\ss,0)}]
\coordinate (A1) at  (0,5);
\coordinate (A2) at  (0,10);
\coordinate (B1) at  (4.33,-2.5);
\coordinate (B2) at  (8.66, -5);
\coordinate (C1) at  (-4.33, -2.5);
\coordinate (C2) at  (-8.66, -5);
%%%% 11:20:20
\draw[line width=.5pt] (A1)--(B1);
\draw[line width=.5pt] (A1)--(B2);
\draw[line width=.5pt] (A2)--(B1);
\draw[line width=.5pt] (A2)--(B2);
\draw[double, line width=.5pt, double distance=1.5pt]   (A1)--(C1);
\draw[double, line width=.5pt, double distance=1.5pt]   (A2)--(C2);
\draw[double, line width=.5pt, double distance=1.5pt]   (B1)--(C1);
\draw[double, line width=.5pt, double distance=1.5pt]   (B2)--(C2);
%%%%
\draw[line width=2pt] (A1)--(A2);
\draw[line width=2pt] (B1)--(B2);
\draw[line width=2pt] (C1)--(C2);
\draw[fill=white, draw=black,  thick] (A1) circle (1);
\draw[fill=white, draw=black,  thick] (A2) circle (1);
\draw[fill=white, draw=black,  thick] (B1) circle (1);
\draw[fill=white, draw=black,  thick] (B2) circle (1);
\draw[fill=white, draw=black,  thick] (C1) circle (1);
\draw[fill=white, draw=black,  thick] (C2) circle (1);
\draw (0,-9) node {size $90720$};
\end{scope}
\begin{scope}[shift={(0,-27)}]
\coordinate (A1) at  (0,5);
\coordinate (A2) at  (0,10);
\coordinate (B1) at  (4.33,-2.5);
\coordinate (B2) at  (8.66, -5);
\coordinate (C1) at  (-4.33, -2.5);
\coordinate (C2) at  (-8.66, -5);
%%%% 20:20:20 two cycles
\draw[double, line width=.5pt, double distance=1.5pt]   (A1)--(B1);
\draw[double, line width=.5pt, double distance=1.5pt]   (A2)--(B2);
\draw[double, line width=.5pt, double distance=1.5pt]   (A1)--(C1);
\draw[double, line width=.5pt, double distance=1.5pt]   (A2)--(C2);
\draw[double, line width=.5pt, double distance=1.5pt]   (B1)--(C1);
\draw[double, line width=.5pt, double distance=1.5pt]   (B2)--(C2);
%%%%
\draw[line width=2pt] (A1)--(A2);
\draw[line width=2pt] (B1)--(B2);
\draw[line width=2pt] (C1)--(C2);
\draw[fill=white, draw=black,  thick] (A1) circle (1);
\draw[fill=white, draw=black,  thick] (A2) circle (1);
\draw[fill=white, draw=black,  thick] (B1) circle (1);
\draw[fill=white, draw=black,  thick] (B2) circle (1);
\draw[fill=white, draw=black,  thick] (C1) circle (1);
\draw[fill=white, draw=black,  thick] (C2) circle (1);
\draw (0,-9) node {size $1120$};
\end{scope}
\begin{scope}[shift={(\s,-27)}]
\coordinate (A1) at  (0,5);
\coordinate (A2) at  (0,10);
\coordinate (B1) at  (4.33,-2.5);
\coordinate (B2) at  (8.66, -5);
\coordinate (C1) at  (-4.33, -2.5);
\coordinate (C2) at  (-8.66, -5);
%%%% 20:20:20 one cycles
\draw[double, line width=.5pt, double distance=1.5pt]   (A1)--(B1);
\draw[double, line width=.5pt, double distance=1.5pt]   (A2)--(B2);
\draw[double, line width=.5pt, double distance=1.5pt]   (A1)--(C1);
\draw[double, line width=.5pt, double distance=1.5pt]   (A2)--(C2);
\draw[double, line width=.5pt, double distance=1.5pt]   (B1)--(C2);
\draw[double, line width=.5pt, double distance=1.5pt]   (B2)--(C1);
%%%%
\draw[line width=2pt] (A1)--(A2);
\draw[line width=2pt] (B1)--(B2);
\draw[line width=2pt] (C1)--(C2);
\draw[fill=white, draw=black,  thick] (A1) circle (1);
\draw[fill=white, draw=black,  thick] (A2) circle (1);
\draw[fill=white, draw=black,  thick] (B1) circle (1);
\draw[fill=white, draw=black,  thick] (B2) circle (1);
\draw[fill=white, draw=black,  thick] (C1) circle (1);
\draw[fill=white, draw=black,  thick] (C2) circle (1);
\draw (0,-9) node {size $30240$};
\end{scope}
\end{tikzpicture}
\caption{Orbits in $\barL(X_b)^{\{3\}}$}\label{fig:orbitsbarLX3}
\end{figure}
%
%%%%%%%%%%%%%%%%%%%%%%%%%%%%%%%%%%%%%%
%
\subsection{Lines in the birational model \texorpdfstring{${\bfP}^{2}$}{bP2}}\label{subsec:linesinbP2}
We choose  disjoint $8$ lines $l_1, \dots, l_8$ in $X_b$
and consider the contraction 
$\beta\colon X_b\to {\bfP}^{2}$
of these lines.
Let $p_i\in {\bfP}^{2}$ be the point  $\beta(l_i)$ for $i=1, \dots, 8$,
and let $h\in \Pic(X_b)$ be the class of the pullback of a line in ${\bfP}^{2}$.
Since
 \[
 3h=(\alpha_b+l_1+ \dots+l_8),
\]
and we have $\ell+\involB(\ell)=2\alpha_b$, it follows that
\[
\intf{h, \ell}+\intf{h, \involB(\ell)}=6
\]
holds for any line $\ell$.
We investigate the images of the lines $\ell$ by $\beta$.
Calculating the intersection numbers with the exceptional lines $l_1, \dots, l_8$ of $\beta$,
we obtain the following.
(See also~\cite[Theorem~26.2]{Manin1986}.)
In the following, the phrase ``$C$ passes through $p\in {\bfP}^{2}$ \emph{once}"
means that $p$  is a smooth point of $C$. 
\begin{itemize}[itemsep=5pt]
\item There exist exactly $8$ lines $l_i$ with $h$-degree $0$. 
Their $\involB$-partners are of $h$-degree $6$:
the sextic curve  $\beta(\involB(l_i))\subset {\bfP}^{2}$  has a triple  point at $p_i$, 
and double points at the $7$ points in $\{p_1, \dots, \ p_8\}\setminus\{p_i\}$.
\item There exist exactly $28$ lines $l_{ij}$ with $h$-degree $1$,
where $1\le i<j\le 8$.
The line $l_{ij}$ is mapped by $\beta$ to the line  in ${\bfP}^{2}$ 
passing through  $p_i$ and $p_j$.
Their $\involB$-partners are  of   $h$-degree $5$:
the quintic curve  $\beta(\involB(l_{ij}))\subset {\bfP}^{2}$ passes through $p_i$ and $p_j$ once, 
and has  double points at the $6$ points  in 
$\{p_1, \dots, \ p_8\}\setminus\{p_i, p_j\}$.
\item There exist exactly $56$ lines $l_{\overline{ijk}}$ with $h$-degree $2$,
where $1\le i<j<k\le 8$.
The line $l_{\overline{ijk}}$  is  mapped by $\beta$ to the conic   in ${\bfP}^{2}$ 
passing through the five points 
in $\{p_1, \dots, p_8\}\setminus \{p_i, p_j, p_k\}$.
Their $\involB$-partners are  of $h$-degree $4$:
the quartic curve  $\beta(\involB(l_{\overline{ijk}}))$
has  double points at 
$p_i, p_j, p_k$ and 
 passes through the $5$ points  in $\{p_1, \dots, p_8\}\setminus \{p_i, p_j, p_k\}$ once.
\item There exist exactly $56$ lines $l_{i, j}$
with $h$-degree $3$,
where $1\le i, j\le 8$ and $i\ne j$.
We have $\involB(l_{i, j})=l_{j, i}$.
The cubic curve  $\beta(\involB(l_{i, j}))$
  passes through the $6$ points in 
$\{p_1, \dots, \ p_8\}\setminus\{p_i, p_j\}$ once, 
has a double point at $p_i$, and does not pass through $p_j$.
\end{itemize}
%
%See Table~\ref{table:hdegandmults} for the summary of the curves $\beta(\ell)$.
%
%[ [ [ 0, [ [ -1, 1 ], [ 0, 7 ] ] ], 8 ],
% [ [ 1, [ [ 0, 6 ], [ 1, 2 ] ] ], 28 ],
%  [ [ 2, [ [ 0, 3 ], [ 1, 5 ] ] ], 56 ], 
%  [ [ 3, [ [ 0, 1 ], [ 1, 6 ], [ 2, 1 ] ] ], 56 ], 
%  [ [ 4, [ [ 1, 5 ], [ 2, 3 ] ] ], 56 ], 
%  [ [ 5, [ [ 1, 2 ], [ 2, 6 ] ] ], 28 ], 
%  [ [ 6, [ [ 2, 7 ], [ 3, 1 ] ] ], 8 ] ] 
 %
 \begin{table}
 \[
 \begin{array}{ccc}
 \textrm{$h$-degree $\intf{h, \ell}$}  & \textrm{multiplicities $\intf{l_i, \ell}$} & \textrm{number}\\
 \hline
 0 & (-1)^1 0^7 & 8 \mystruth{12pt}\\
1 & 0^6 1^2 & 28 \\
2 &  0^3 1^5 & 56 \\
3 & 0^1 1^6 2^1 & 56 \\
4 & 1^5 2^3 & 56 \\
5 & 1^2 2^6 & 28 \\
6 & 2^7 3^1& 8 
 \end{array}
 \]
 \caption{$240$ curves in \texorpdfstring{${\bfP}^{2}$}{P2}}\label{table:hdegandmults}
 \end{table}
 %
%\begin{remark}
%If we are given the coordinates of the points $p_1, \dots, p_8$
%satisfying the conditions in Definition~\ref{def:PPPn},
%then we can calculate the defining equations of the images  $\beta(\ell)$
% by solving \emph{linear} equations.
%\end{remark}
%
%\begin{remark}
%In~\cite{Degtyarev2012Bertini},
%an explicit description of 
%the Bertini involution $\involB$ on $\bP^2$  is given.
%\end{remark}
 %
\subsection{Union of lines}\label{subsec:general}
In this section, we confirm the following result,
which must be well known,  but of which we could not find a reference.
\begin{proposition}\label{prop:union}
The union of the $240$ lines  in $X_b$ has only ordinary double points 
as its singularities.
\end{proposition}
\begin{proof}
Recall that $X_b$ is a \emph{general} member of the family $\XXX\to \UUU$.
By Corollary~\ref{cor:irred},
it is enough to prove the following.
\begin{itemize}
\item[($m$)] For $m=2$ and $m=3$,
there exist a point $u\in \UUU$
and  lines $\ell_1, \ell_2$ in $X_u$ such that
$\intf{\ell_1, \ell_2}=m$, and that 
$\ell_1$ and $\ell_2$ intersect at distinct $m$ points.
\item[(\,$t$\,)]
For each $t=[t_1, t_2, t_3]$ in the second line of~\eqref{eq:tenlist},
there exist a point $u\in \UUU$
and  lines $\ell_1, \ell_2, \ell_3$ in $X_u$ such that $[\intf{\ell_1, \ell_2}, \intf{\ell_2, \ell_3},  \intf{\ell_1, \ell_3}]$ 
is equal to $t$ up to order,
and that  $\ell_1\cap  \ell_2\cap  \ell_3$  is empty.
\end{itemize}
We find such a del Pezzo surface $X_u$ by choosing $8$ points 
$p_1, \dots, p_8$ on ${\bfP}^{2}$ satisfying  the conditions in Definition~\ref{def:PPPn}.
Let $Y_{\bp}\to {\bfP}^{2}$ be the blowing-up at these points.
The lines in  $Y_{\bp}$ can be calculated 
by the description given in Section~\ref{subsec:linesinbP2},
and we search for lines satisfying
 the conditions  in  ($m$) and ($t$).
It is enough to find an example over a finite field.
 \par
 We give an example over  $\FF_{19}$.
 We choose the following 
$8$ points:
% newexamples19[1];
%[ [ 0*Z(19), 0*Z(19) ], [ Z(19)^0, 0*Z(19) ], [ 0*Z(19), Z(19)^0 ], [ Z(19)^0, Z(19)^0 ],
% [ Z(19), Z(19)^11 ], [ Z(19)^11, Z(19)^2 ], [ Z(19)^12, Z(19)^11 ],   [ Z(19)^15, Z(19)^4 ] ]
%gap> Z(19)-2*One(GF(19));
%0*Z(19)
%gap> IsZero(Z(19)-2*One(GF(19)));
%true
%gap> [2,2^11] mod 19;
%[ 2, 15 ]
%gap> [2^11,2^2] mod 19;
%[ 15, 4 ]
%gap> [2^12, 2^11] mod 19;
%[ 11, 15 ]
%gap> [2^15, 2^4] mod 19;
%[ 12, 16 ]
 \[
 \begin{aligned}
&p_1=(0,0),
\quad
p_2=(1,0),
\quad
p_3=(0,1),
\quad
p_4=(1,1),\\
\quad
&p_5=(2,15),
\quad
p_6=(15,4),
\quad
p_7=(11,15),
\quad
p_8=(12, 16),
\end{aligned}
 \]
where we use  affine coordinates of ${\bfP}^{2}$.
It is easy to confirm that these points 
satisfy  the conditions in Definition~\ref{def:PPPn}.
A line $\ell$ in $Y_{\bp}$ is denoted as $[d; \mu_1, \dots, \mu_8]$,
where
$d$ is the $h$-degree and $\mu_i$ is the multiplicity $\intf{l_i, \ell}$ of $\beta(\ell)$ at $p_i$.
%Note that, if $\beta(\ell)=\{p_i\}$, then we put $\mu_i=-1$.
We consider the following lines ${\ell}_i$, and calculate the defining equations of $\beta({\ell}_i)$ in ${\bfP}^{2}$:
\[
\begin{array}{lcl}
{\ell}_1 &:=& [0;  -1, 0, 0, 0, 0, 0, 0, 0 ], \\ %[ -1, 0, 0, 0, 0, 0, 0, 0 ] 1 
{\ell}_2 &:=& [3;  2, 1, 1, 1, 1, 1, 1, 0 ], \\ %[ 2, 1, 1, 1, 1, 1, 1, 0 ] 148 
{\ell}_3 &:=& [6;  3, 2, 2, 2, 2, 2, 2, 2 ], \\  %[ 3, 2, 2, 2, 2, 2, 2, 2 ] 240 
{\ell}_4 &:=& [1; 1, 1, 0, 0, 0, 0, 0, 0 ], \\ %[ 1, 1, 0, 0, 0, 0, 0, 0 ] 36 
{\ell}_5 &:=& [2; 1, 0, 1, 1, 1, 1, 0, 0 ], \\ %[ 1, 0, 1, 1, 1, 1, 0, 0 ] 88 
\end{array}
\quad
\begin{array}{lcl}
{\ell}_6 &:=& [3; 2, 0, 1, 1, 1, 1, 1, 1 ], \\ %[ 2, 0, 1, 1, 1, 1, 1, 1 ] 121 
{\ell}_7 &:=& [2; 1, 1, 1, 1, 1, 0, 0, 0 ], \\%[ 1, 1, 1, 1, 1, 0, 0, 0 ] 92 
{\ell}_8 &:=& [6; 2, 3, 2, 2, 2, 2, 2, 2 ], \\% [ 2, 3, 2, 2, 2, 2, 2, 2 ] 239 
{\ell}_9 &:=& [6; 2, 2, 3, 2, 2, 2, 2, 2 ]. \\%[ 2, 2, 3, 2, 2, 2, 2, 2 ] 238
&&
\end{array}
\]
%Then we confirm that 
%\begin{itemize}
%\item the pair $l_1, l_2$ satisfies condition $(m)$ for $m=2$,  %[ 1, 148 ] [ 2 ] 
%\item the pair $l_1, l_3$ satisfies condition $(m)$ for $m=3$, %[ 1, 240 ] [ 3 ] 
%\item the triple $l_1, l_4, l_5$ satisfies condition $(t)$ for $t=[1,1,1]$, %[ 1, 36, 88 ] [ 1, 1, 1 ] 
%\item the triple $l_1, l_4, l_6$ satisfies condition $(t)$ for $t=[1,1,2]$, %[ 1, 36, 121 ] [ 1, 1, 2 ] 
%\item the triple $l_1, l_7, l_3$ satisfies condition $(t)$ for $t=[1,1,3]$, %[ 1, 92, 240 ] [ 1, 1, 3 ] 
%\item the triple $l_1, l_6, l_9$ satisfies condition $(t)$ for $t=[1,2,2]$,  %[ 1, 121, 238 ] [ 1, 2, 2 ] 
%\item the triple $l_1, l_6, l_8$ satisfies condition $(t)$ for $t=[2,2,2]$. %[ 1, 121, 239 ] [ 2, 2, 2 ] 
%\end{itemize}
%
Then we confirm that the pair ${\ell}_1, {\ell}_2$ (resp.~the pair  ${\ell}_1, {\ell}_3$) satisfies condition $(m)$ for $m=2$
 (resp.~$m=3$).
Moreover, 
the triple $\tau=\{{\ell}_i, {\ell}_j, {\ell}_k\}$ satisfies condition $(t)$ for $t=[t_1, t_2, t_3]$,
where
\[
\begin{array}{cc}
\tau & t \\
\hline 
{\ell}_1, {\ell}_4, {\ell}_5 & [1,1,1]\\
{\ell}_1, {\ell}_4, {\ell}_6 & [1,1,2]\\
{\ell}_1, {\ell}_3, {\ell}_7 & [1,1,3]
\end{array}
\quad
\begin{array}{cc}
\tau & t \\
\hline 
{\ell}_1, {\ell}_6, {\ell}_9& [1,2,2]\\
{\ell}_1, {\ell}_6, {\ell}_8 & [2,2,2]\rlap{\;.}\\
 & 
\end{array}
\]
%
%gap> FindPosByMults([-1,0,0,0,0,0,0,0]);
%[ -1, 0, 0, 0, 0, 0, 0, 0 ] 1 
%gap> FindPosByMults([2,1,1,1,1,1,1,0]);
%[ 2, 1, 1, 1, 1, 1, 1, 0 ] 148 
%gap> FindPosByMults([3,2,2,2,2,2,2,2]);
%[ 3, 2, 2, 2, 2, 2, 2, 2 ] 240 
%gap> FindPosByMults([1,1,0,0,0,0,0,0]);
%[ 1, 1, 0, 0, 0, 0, 0, 0 ] 36 
%gap> FindPosByMults([1,0,1,1,1,1,0,0]);
%[ 1, 0, 1, 1, 1, 1, 0, 0 ] 88 
%gap> FindPosByMults([2,0,1,1,1,1,1,1]);
%[ 2, 0, 1, 1, 1, 1, 1, 1 ] 121 
%gap> FindPosByMults([1,1,1,1,1,0,0,0]);
%[ 1, 1, 1, 1, 1, 0, 0, 0 ] 92 
%gap> FindPosByMults([2,3,2,2,2,2,2,2]);
%[ 2, 3, 2, 2, 2, 2, 2, 2 ] 239 
%gap> FindPosByMults([2,2,3,2,2,2,2,2]);
%[ 2, 2, 3, 2, 2, 2, 2, 2 ] 238 
%gap> IntNumbs([1,148]);
%[ 1, 148 ] [ 2 ] 
%gap> IntNumbs([1,240]);
%[ 1, 240 ] [ 3 ] 
%gap> IntNumbs([1,36,88]);
%[ 1, 36, 88 ] [ 1, 1, 1 ] 
%gap> IntNumbs([1,36,121]);
%[ 1, 36, 121 ] [ 1, 1, 2 ] 
%gap> IntNumbs([1,92,240]);
%[ 1, 92, 240 ] [ 1, 1, 3 ] 
%gap> IntNumbs([1,121,238]);
%[ 1, 121, 238 ] [ 1, 2, 2 ] 
%gap> IntNumbs([1,121,239]);
%[ 1, 121, 239 ] [ 2, 2, 2 ] 
%
%
The defining equations of $\beta({\ell}_i)$ can be found in~\cite{WE8compdata}.
\end{proof}
%
%The following will be used in Section~\ref{sec:family}.
%
\begin{corollary}\label{cor:general1}
{\rm (1)}
Every tangent plane section $H\cap Q$  for $B_b$ intersects $B_b$ at distinct  three  points.
{\rm (2)} 
Suppose that  $H_1\cap Q$ and $H_2\cap Q$ are distinct  tangent plane sections for $B_b$.
Then  $H_1\cap B_b$ and $ H_2\cap B_b$ are disjoint, 
 and  $H_1\cap H_2\cap Q$ consists of distinct two points.
 {\rm (3)} 
 Suppose that  $H_1\cap Q, H_2\cap Q$, and $H_3\cap Q$ are distinct  tangent plane sections for $B_b$.
Then   $H_1\cap H_2\cap H_3 \cap Q$ is empty.
\qed
\end{corollary}
%
%
%\begin{remark}
%In~\cite{DesjardinsWinter2025} and \cite{LuijkWinter2023},
%possible singular points on the union of the  $240$ lines in
%del Pezzo surfaces  of degree $1$ other than ordinary double points  
%(\emph{generalized Eckardt points}) are extensively studied.
%\end{remark}
%
%%%%%%%%%%%%%%%%%
%
%
\section{Del Pezzo surfaces of degree one and \texorpdfstring{$\mtan{3}$}{mtan3}-sextics}\label{sec:mtan3anddelPezzo}
In this section,
we relate  $\mtan{3}$-sextics
with del Pezzo surfaces of degree $1$.
In Section~\ref{subsec:planecurveswithmtan},
we prove Propositions~\ref{prop:t3}
and exhibit the parameter space $\TTT$
of $\mtan{3}$-sextics in the frame $(A, \Lambda)$.
In Section~\ref{subsec:pip},
we describe a birational map $\pi_p$ from 
$Q$ to $\PP^2$,
which gives a proof of Propositions~\ref{prop:t3conics},
and induces a birational map between $\UUU$ and $\TTT$.
\subsection{Plane curves with \texorpdfstring{$\mtan{m}$}{tm}-singularity}\label{subsec:planecurveswithmtan}
We fix a point $A\in \PP^2$
and a line $\Lambda\subset \PP^2$ passing through $A$.
Let $(x, y)$ be affine coordinates 
of $\PP^2$ such that $A=(0,0)$ and $\Lambda=\{y=0\}$.
We consider a plane curve  $C\subset\PP^2$ of degree $d$
defined by 
\[
f(x, y)=\sum_{\mu+\nu\le d} a_{\mu\nu} x^{\mu} y^{\nu}=0.
\]
\begin{proposition}\label{prop:tm}
Suppose that $d\ge 2m$.
The plane curve $C=\{f=0\}$ has a $\mtan{m}$-singularity 
at $A$ with the tangent line $\Lambda$ if and only if the following holds:
\begin{enumerate}[label={\rm (\roman*)}]
\item $a_{\mu\nu}=0$ if $\mu+2\nu<2m$, and 
\item the following equation has distinct $m$ roots:
\[
a_{0, m}\, z^m+a_{2, m-1} \, z^{m-1} + \dots +a_{2m-2, 2}\, z+ a_{2m, 0}=0.
\]
%with variable $z$ has distinct $m$ roots.
\end{enumerate}
\end{proposition}
\begin{proof}
We consider the blowing-up 
\[
(u, v)\;\; \mapsto\;\; (x, y)=(u, uv)
\]
of $\PP^2$ at $A$.
The strict transform of $\Lambda$ is given by $v=0$,
and it intersects with the exceptional divisor $E=\{u=0\}$
at the point $(u, v)=(0,0)$.
Then $C$ has a $\mtan{m}$-singularity 
at $A$ with the tangent line $\Lambda$ if and only if 
the  total transform
\[
f(u, uv)=0
\]
of $C$ contains $E$ with multiplicity $m$
and the strict transform
\[
\sum_{\mu+\nu\le d} a_{\mu\nu} u^{\mu+\nu -m} v^{\nu}=0
\]
of $C$ has an ordinary $m$-fold point at $(u, v)=(0,0)$
with each local branch intersecting $E$ transversely.
This  condition is  equivalent to conditions (i) and (ii) 
in the statement.
\end{proof}
Proposition~\ref{prop:t3} follows from Proposition~\ref{prop:tm} immediately.
In the following, 
for a point $t\in \TTT$, we denote by $C_t\subset\PP^2$ the corresponding $\mtan{3}$-sextic.
\subsection{Birational map \texorpdfstring{$\pi_p$}{pip}}
\label{subsec:pip}
Recall that 
$Q\subset \PP^3$ is  a quadric cone %a singular quadric surface of rank $3$
with the vertex $V\in Q$.
Then $Q$ is ruled by lines passing through $V$.
We choose a smooth point $p\in Q\setminus\{V\}$.
Then the projection from $p$ induces a birational map
\[
\pi_p\colon Q\birat \PP^2.
\]
This birational map $\pi_p$ defines a frame $(A, \Lambda)$ in $\PP^2$ as follows.
\par
We describe $\pi_p$ in detail.
Let $r_p\subset Q$ be  the line in the ruling passing through $p$, 
and let  $\tilQ\to Q$ be  the composite of the minimal desingularization of $Q$
and the blowing-up at $p$.
Let $E_V$, $E_p$, and $\tilde{r}_p$ be the exceptional curve over $V$,
the exceptional curve over $p$, and the strict transform of $r_p$ in $\tilQ$.
Then $E_V$, $E_p$, $\tilde{r}_p$ are  smooth rational curves on $\tilQ$
with self-intersection number $-2$, $-1$, $-1$, respectively.
We blow down $\tilde{r}_p$ to a point, and then 
we blow down the image $E_V\sprime$ of $E_V$ to a point.
The resulting surface is 
the target  plane $\PP^2$ of $\pi_p$.
The  curves $E_V$ and  $\tilde{r}_p$  are mapped to a point $A\in \PP^2$,
and the curve $E_p$ is mapped to a line $\Lambda$ passing through $A$.
Every member of the ruling of $Q$ other than $r_p$ is mapped to a line  passing through $A$.
A plane section $H\cap Q$ not containing $V$ and $p$ is mapped to a smooth conic 
that is  passing through $A$ and is tangent to $\Lambda$ at $A$.
\par
The inverse of $\pi_p$ is given as follows.
Let $(\PP^2)^{\mathord{\sim}}\to \PP^2$ be the blowing up at $A$,
and let $\tilQ\to (\PP^2)^{\mathord{\sim}}$  be the blowing up at the intersection point $A\sprime$
of the exceptional curve $E_A$ over $A$ and the strict transform of $\Lambda$.
Then the strict transform $\widetilde{E}_A$ of $E_A$ (resp.~$\widetilde{\Lambda}$ of $\Lambda$) in $\tilQ$ 
is a smooth rational curve of self-intersection number $-2$ (resp.~$-1$).
We contract these two curves, and 
let $\tilQ\to Q$ denote  the contraction map.
Then $\widetilde{E}_A$  is contracted to the singular point $V$ of $Q$,  
and $\widetilde{\Lambda}$ is contracted to the  center $p$
of the projection.
The exceptional curve $E_{A\sprime}\subset \tilQ$ over $A\sprime$ is 
mapped to the line $r_p$ in the ruling.  
\begin{figure}
\begin{tikzpicture}[x=.4mm, y=.4mm]
\begin{scope}[scale=.9]
\draw[thick, rotate=-45] (0,0) ellipse [x radius=30, y radius=15];
\draw[thick] (-70, -70)--(-10.6, -10.6);
\draw[thick] (-70, -70)--(-22.8, 18.9);
\draw[thick] (-70, -70)--(18.9, -22.8);
\filldraw[black] (-20,-20) circle[radius=1.5];
\filldraw[black] (-69, -69) circle[radius=1.5];
\draw (-10.6, -9.5) node[right]{$r_p$};
\draw (-69.5, -65) node[above]{$V$};
\draw (-20,-22) node[below]{$p$};
\end{scope}
\begin{scope}[shift={(60, -5)}, scale=1.2]
\draw[->, very thick] (70,20)--(85, 5);
\filldraw[black] (90, -10) circle[radius=1.5];
\draw[thick] (60, -10)--(120, -10);
\draw (90, -10) node[below]{$A$};
\draw (120, -10) node[right]{$\Lambda$};
\end{scope}
\begin{scope}[shift={(-5,0)}, scale=.9]
\draw[<-, very thick] (20,25)--(35, 40);
\draw[thick] (50, 100)--(120, 100);
\draw[thick] (50, 60)--(75, 120);
\draw[thick] (120, 60)--(95, 120);
\draw(120, 100) node[right]{$\tilde{r}_p=E_{A\sprime}$};
\draw(50, 60) node[below]{$E_V=\widetilde{E}_A$};
\draw(120, 60) node[below]{$E_p=\widetilde{\Lambda}$};
\draw(120, 125) node[above]{\phantom{aaa}};
\end{scope}
\end{tikzpicture}
\caption{Birational map $\pi_p$}\label{fig:birat}
\end{figure}
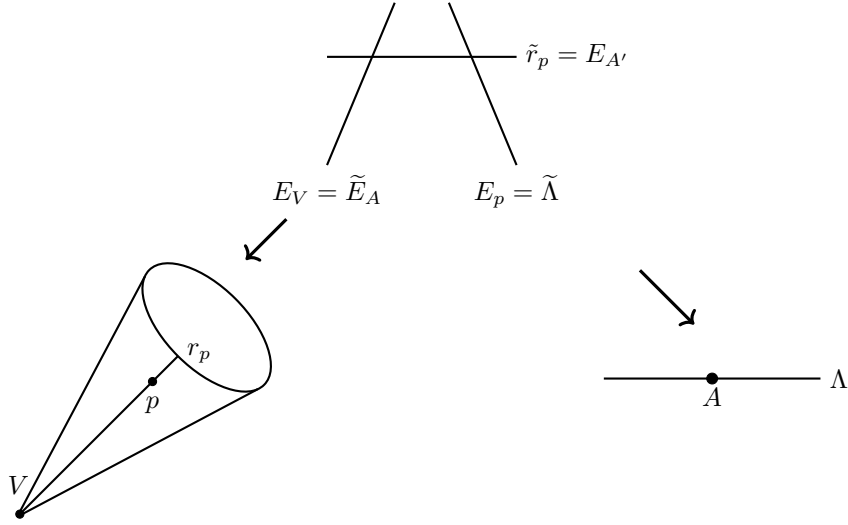
\subsection{Birational map \texorpdfstring{$\Pi_p$}{Pip}}
\label{subsec:Pip}
Recall that  $\UUU\subset |\OOO_{Q}(3)|$ is the parameter space of
the family $\XXX\to \UUU$ of bi-anti-canonical models of del Pezzo surfaces of degree $1$, and that 
 $\TTT$ is the parameter space of the family of $\mtan{3}$-sextics $C_t\subset \PP^2$ 
in the frame $(A, \Lambda)$.
Note that both of $\TTT$ and $\UUU$ are of dimension $15$.
\par
We have chosen a point $p\in Q\setminus\{V\}$.
Suppose that $u$ is a general point of $\UUU$.
In particular, the curve $B_u$ does not contain $p$ and the ruling line $r_p$ intersects $B_u$ 
at three distinct points.
Then the image of  $B_u$ by $\pi_p$ is a $\mtan{3}$-sextic $C_{t(u)}$ in the frame $(A, \Lambda)$,
where $t(u)\in \TTT$.
Conversely, if $t$ is a general point of $\TTT$,
then the image of the $\mtan{3}$-sextic $C_t$ by the inverse of $\pi_p$ is a member $B_{u(t)}$ of $\UUU$.
Therefore
the birational map
$\pi_p$ between $Q$ and $\PP^2$ induces a birational map %between $\UUU$ and $\TTT$.
\begin{equation}\label{eq:biratUUUTTT}
\Pi_p\colon \UUU\birat\TTT.
\end{equation}
%
%This birational map plays an important role in 
%the proof of  Theorem~\ref{thm:main2} in 
%Section~\ref{sec:embtop}.
%
\begin{proof}[Proof of Proposition~\ref{prop:t3conics}]
Recall that
$G(C_t)$ denotes the set of special tangent conics 
of  the $\mtan{3}$-sextic $C_t$ for $t\in \TTT$, and that
 $S(B_u)$ denotes the set of tangent plane sections for $B_u$ for $u\in \UUU$.
Suppose that $u\in \UUU$ is  general.
Then the  birational map
$\pi_p$ induces a bijection
\[
S(B_u)\;\cong\;G(C_{t(u)}).
\]
Moreover, 
the properties of the tangent plane sections for $B_u$ 
given in Corollary~\ref{cor:general1} carry over to properties
of the special tangent conics of $C_t$,
ensuring that they are in a general position.
\end{proof}
 \section{Embedding topology}\label{sec:embtop}
 In this section, we prove Theorem~\ref{thm:main2}.
 The following observation is due to Artal~Bartolo.
 \begin{lemma}\label{lem:AB}
 Any self-homeomorphism of $\PP^2$ preserves the orientation.
 \end{lemma}
  \begin{proof}
  Note that the intersection form on $H_2(\PP^2)$ 
  depends on the choice of an orientation.
  For an orientation $\xi$ of $\PP^2$, let $\intf{\phantom{i,i}}_{\xi}$ denote 
  the corresponding  intersection form on $H_2(\PP^2)$.
  Then we have $\intf{\phantom{i,i}}_{-\xi}=-\intf{\phantom{i,i}}_{\xi}$.
  Let $\eta$ be the orientation coming from the complex structure of $\PP^2$,
which satisfies $\intf{\gamma, \gamma}_{\eta}\ge 0$
  for any  $\gamma\in H_2(\PP^2)$.
Suppose that a self-homeomorphism $f$ of $\PP^2$ satisfies $f_*\eta = -\eta$.
Then we have 
  \[
  \intf{\gamma,\gamma}_{\eta}=  \intf{f_* \gamma,f_* \gamma}_{f_* \eta}=- \intf{f_* \gamma,f_* \gamma}_{\eta}, 
  \]
which is a contradiction.
  \end{proof}
  Now we prove our main result.
 \begin{proof}[Proof of Theorem~\ref{thm:main2}]
 Let $\UUU$ and $\TTT$ be as in Section~\ref{subsec:Pip}.
 We choose Zariski open subsets $\TTT^0\subset \TTT$ and $\UUU^0\subset \UUU$
such that $\TTT^0$ and $\UUU^0$ are isomorphic via the birational map
$\Pi_p\colon \UUU\birat \TTT$,  
and such that
the special tangent conics of $C_t$ are 
in general position for any $t\in \TTT_0$.
We fix a  point $o\in \TTT^0$, 
and let $b\in \UUU^0$ be the  point corresponding to $o$ via $\Pi_p$. 
Consider the family  $\GGG^0\to \TTT^0$
whose fiber over $t\in  \TTT^0$ is the set  $G(C_t)$ of special tangent conics. 
Then $\GGG^0\to \TTT^0$ is 
isomorphic to the pullback 
of the family $\SSS \to \UUU$
of the sets $S (B_u)$ of tangent plane sections 
via the morphism
\begin{equation}\label{eq:TTT0UUU}
\TTT^0\;\;\cong\;\; \UUU^0\;\;\inj\;\; \UUU.
\end{equation}
In particular,
the monodromy action 
of $\pione(\TTT^0, o)$ on $G(C_o)$ is induced 
by
the monodromy action 
of $\pione(\UUU, b)$ on $S (B_b)$
via the bijection $S (B_b)\cong G(C_o)$ given by $\pi_p$ 
and the surjective homomorphism
\[
\pione(\TTT^0, o)\surj \pione(\UUU, b)
\]
induced by~\eqref{eq:TTT0UUU}.
Recall from Proposition~\ref{prop:WE8} that,
under the identification of $S (B_b)$ with  $\barDelta(R(X_b))$ in~\eqref{eq:barLX}, 
 the monodromy action 
of $\pione(\UUU, b)$ on $S (B_b)$
factors through 
the surjective homomorphism~\eqref{eq:mon}
to the Weyl group $W(R(X_b))$. % of type $E_8$.
\par
We  consider 
the family $\GGG\sp{0\{k\}} \to \TTT^0$
whose fiber over $t$ is the set $G(C_t)\spset{k}$.
By the above discussion on the monodromy and the definition of $N(k)$, 
the space  $\GGG\sp{0\{k\}}$  has exactly $N(k)$ connected components.
If two points $s, s\sprime$
of $G(C_o)\spset{k}$ are in the same connected component 
of $\GGG\sp{0\{k\}}$,
then we can deform $D_{o,s}$ to   $D_{o, s\sprime}$ in $\PP^2$
without changing the embedding topology
along a path in $\GGG\sp{0\{k\}}$ connecting $s$ and $s\sprime$.
\par
To complete the proof of Theorem~\ref{thm:main2}, it is enough to show that, 
if $D_{o,s}$ and $D_{o, s\sprime}$ have the same embedding topology,
then $s$ and $s\sprime$ belong to the same connected component 
of $\GGG\sp{0\{k\}}$.
For a special tangent conic $\Gamma\in G(C_o)$ of $C_o$,
let $\delta_{\Gamma}\in \barDelta(R(X_b))$ 
denote the pair $\{[l_{\Gamma}]_R, -[l_{\Gamma}]_R\}\subset R(X_b)$,
where $\{l_{\Gamma}, \involB(l_{\Gamma})\}$ is the $\involB$-pair  obtained from 
the tangent plane section 
%$H_{\Gamma}\cap Q$ 
for $B_b$
corresponding to $\Gamma$ via $\pi_p$.
We then put
\[
\delta(s):=\set{\delta_{\Gamma}}{\Gamma\in s}\;\;\in\;\; \barDelta(R(X_b))\spset{k}.
\]
To show that $s$ and $s\sprime$ are in the same connected component of $\GGG\sp{0\{k\}}$,
it is enough to 
find an isometry $g\in W(R(X_b))$ of the lattice $R(X_b)$
such that
the self-bijection of $\barDelta(R(X_b))\spset{k}$ induced by $g$ 
maps $\delta(s)$ to $\delta(s\sprime)$.
\par
Suppose that  $D_{o,s}$ and $D_{o, s\sprime}$ have the same embedding topology.
We have a homeomorphism
\[
\Psi\;\;\colon\;\; (\PP^2, D_{o, s})\;\isom\; (\PP^2, D_{o, s\sprime}).
\]
Then $\Psi$ induces a self-homeomorphism $\Psi_M$ of the complement 
\[
M_o:=\PP^2-(C_o+\Lambda).
\]
Since $H_1(M_o)\cong \ZZ$,
there exists a unique double covering 
\[
\varphi_M\colon Z_o\to M_o
\]
of $M_o$
by a connected surface $Z_o$, and  $\Psi_M$ 
 lifts to a self-homeomorphism $\Psi_Z$ of $Z_o$.
 Note that the lift $\Psi_Z$ is unique up to the deck-transformation of $Z_o$ over $M_o$.
Since $\Psi_M$ is the restriction of a self-homeomorphism of $\PP^2$,
Lemma~\ref{lem:AB} implies that   $\Psi_M$
preserves the orientation of $M_o$, 
and hence  $\Psi_Z$
is an orientation-preserving homeomorphism of $Z_o$.
Consequently, the automorphism 
\[
g_Z(\Psi)\colon H_2(Z_o)\isom H_2(Z_o)
\]
of the $\ZZ$-module $H_2(Z_o)$
induced by $\Psi_Z$  
preserves the intersection form $\intf{\phantom{i,i}}_Z$
given  by the complex structure of $Z_o$.
We put
\[
\Ker\, \intf{\phantom{i,i}}_Z:=\set{x\in H_2(Z_o)}{\intf{x, y}_Z=0\;\textrm{for any}\; y\in H_2(Z_o)}.
\]
Then $g_Z(\Psi)$ gives rise to  
an isometry  
of the lattice 
\[
\barH_2(Z_o):=H_2(Z_o)/\Ker\, \intf{\phantom{i,i}}_Z.
\]
Note that $\pi_p$ induces an isomorphism from $M_o$ to 
\[
Q_b^0:=Q-(B_b+r_p).
\]
Let $2\tilB_b$ and $(r_p)\sp{\sim}$ be the pullback of $B_b$ and $r_p$ by the double covering
$X_b\to Q$. 
Then the double covering $Z_o$ of $M_o$ can be identified with 
\[
X_b^0:=X_b-(\tilB_b+(r_p)\sp{\sim}).
\]
Note that this identification 
 is unique up to the deck-transformation of $X_b^0$ over~$Q_b^0$.
In the following,  
we regard $Z_o$ as a Zariski open subset of $X_b$.
We also  identify $H_2(X_b)$ with $H^2(X_b)=\Pic(X_b)$ by the Poincar\'e duality.
Since the homology classes of $\tilB_b$ and  $(r_p)\sp{\sim}$  in $H_2(X_b)$ are $3\alpha_b$ and $\alpha_b$,
respectively, and both of $\tilB_b$ and  $(r_p)\sp{\sim}$ are irreducible, 
the Poincar\'e-Lefschetz duality implies  that the inclusion $Z_o\inj X_b$ yields a surjective homomorphism
\begin{equation}\label{eq:iotahom}
H_2(Z_o)\surj (\alpha_b)\sperp \subset H_2(X_b)
\end{equation}
that preserves the intersection form,
where $(\alpha_b)\sperp$ is the orthogonal complement of $\alpha_b$. 
By the identification $H_2(X_b)=\Pic(X_b)$,
we have $(\alpha_b)\sperp=R(X_b)$, and 
hence  the intersection form on $(\alpha_b)\sperp$ is non-degenerate.
Therefore 
the kernel of~\eqref{eq:iotahom} is equal to $\Ker \intf{\phantom{i,i}}_Z$.
In particular, the lattice $\barH_2(Z_o)$ is  isomorphic to the lattice $R(X_b)$.
Consequently,  the homeomorphism $\Psi$ defines an isometry
\[
g_X(\Psi)\;\in\; W(R(X_b)).
\]
Note that $g_X(\Psi)$ is uniquely determined by $\Psi$  up to $\pm \id$.
\par
We will show that 
the self-bijection of $\barDelta(R(X_b))\spset{k}$ induced by 
 $g_X(\Psi)$ 
sends $\delta(s)=\set{\delta_{\Gamma}}{\Gamma\in s}$ to $\delta(s\sprime)$.
Since $\Psi$ maps the elements of $s$ to the elements of $s\sprime$ bijectively,
it suffices to show that,
if $\Psi$ maps a special tangent conic $\Gamma$ to a special tangent conic $\Gamma\sprime$,
then $g_X(\Psi)$ maps $\delta_{\Gamma}$ to $\delta_{\Gamma\sprime}$.
\par
Suppose that $\Psi(\Gamma)=\Gamma\sprime$.
The curve $\varphi_M\inv(\Gamma\cap M_o)$
in $Z_o$ has two connected components $\Gamma^+$ and $\Gamma^-$.
The closure of $\Gamma^+$ in $X_b$ is a line $l_{\Gamma}$,
and the closure of $\Gamma^-$ is its $\involB$-partner.
(Recall that we have $Z_o\subset X_b$.)
These two components, viewed as locally finite topological cycles, 
give rise to
linear forms
\[
\gamma^+\colon H_2(Z_o)\to \ZZ,
\quad 
\gamma^-\colon H_2(Z_o)\to \ZZ, 
\]
by the intersection paring.
By definition, 
each of  these linear forms  factors through the quotient homomorphism 
\[
H_2(Z_o)\surj (\alpha_b)\sperp= R(X_b)
\] 
given by $Z_o\inj X_b$,
and the induced  linear form $(\alpha_b)\sperp\to \ZZ$  coincides with  the intersection pairing with $[l_{\Gamma}]_R\in R(X_b)$
and  with $[\involB(l_{\Gamma})]_R=-[l_{\Gamma}]_R$, respectively.
The connected components of 
the curve $\varphi_M\inv(\Gamma\sprime \cap M_o)$ are
$\Psi_Z(\Gamma^{+})$ and $\Psi_Z(\Gamma^{-})$.
The linear forms
on $H_2(Z_o)$ given by these locally finite topological cycles
are equal to
$\gamma^{+}\circ g_Z(\Psi)\inv$ and $\gamma^{-}\circ g_Z(\Psi)\inv $, respectively.
Each of these linear forms yields
a linear form on $(\alpha_b)\sperp= R(X_b)$,
which is the intersection pairing  with $ [l_{\Gamma\sprime}]_R$ and with $-[l_{\Gamma\sprime}]_R$,
respectively, where $l_{\Gamma\sprime}$ is the closure of $\Psi_Z(\Gamma^{+})$ in $X_b$.
Hence we obtain $g_X(\Psi)( [l_{\Gamma}]_R)= [l_{\Gamma\sprime}]_R$.
\end{proof}
\begin{remark}\label{rem:K3}
The double covering $W$ of $X_b$ branching along a general member of $|2\alpha_b|$
is a $K3$ surface.
In~\cite{Roulleau2022}, it is proved that the automorphism group of $W$ is isomorphic to $ \ZZ/2\ZZ\times\ZZ/2\ZZ$, 
that $W$ contains exactly  $240$ smooth rational curves, and that 
 $W$ has a structure of the double plane $W\to\PP^2$
whose branch curve is a smooth sextic possessing  $120$  conics that are $6$-tangent.
This $K3$ surface had been discovered in~\cite{Kondo1989}.
\end{remark}
\bibliographystyle{plain}

\end{document}